\documentclass[11pt]{article}
\usepackage{amssymb,amsmath}
\usepackage[dvips]{graphicx}
\usepackage{graphicx}
\usepackage{amsthm}

\theoremstyle{plain}
\newtheorem{theorem}{Theorem}
\newtheorem{lemma}[theorem]{Lemma}
\newtheorem{corollary}[theorem]{Corollary}
\newtheorem{proposition}[theorem]{Proposition}

\theoremstyle{definition}
\newtheorem{definition}[theorem]{Definition}
\newtheorem{example}[theorem]{Example}

\theoremstyle{remark}
\newtheorem{remark}[theorem]{Remark}

\begin{document}

\title{Generating families of surface triangulations. The case of punctured surfaces with inner degree at least $4$}

 \author{Mar\'{\i}a-Jos\'e Ch\'avez\footnote{
    Departamento de Matem\'atica Aplicada I,
    Universidad de Sevilla, Spain,
    mjchavez@us.es}
         \and
        Seiya Negami \footnote{Research Institute of Environment and Information Sciences,
        Yokohama National University, 79-2 Tokiwadai, Hodogaya-Ku,
         Yokohama 240-8501, Japan,
         negami@ynu.ac.jp}
        \and Antonio Quintero \footnote{
        Departamento de Geometr\'{\i}a y Topolog\'{\i}a
        Universidad de Sevilla,  Spain,
        quintero@us.es}
        \and
        Mar\'{\i}a Trinidad Villar\footnote{
        Departamento de Geometr\'{\i}a y Topolog\'{\i}a,
        Universidad de Sevilla,  Spain,
        villar@us.es}
}

\date{}

\maketitle

\begin{abstract}

We present two versions of a method for generating all triangulations of any  punctured surface in
each of these two families: (1) triangulations  with inner vertices of degree $\geq 4$ and boundary
vertices of degree $\geq  3$ and (2) triangulations with all vertices of  degree $\geq  4$.
The method is based on a series of reversible operations, termed reductions, which lead to a
minimal set of triangulations in each family. Throughout the process the triangulations remain
within the corresponding family. Moreover, for the family (1) these operations reduce to the
well-known edge contractions and removals of octahedra.
The main results are proved by an exhaustive analysis of all possible local configurations which admit a reduction.


\vspace*{.5cm} This work has been partially supported by
 PAI FQM-164; PAI FQM-189; MTM 2010-20445.

\end{abstract}

\section{Introduction}
By a {\it triangulation} of a surface $F^2$ we mean a simple graph $G$ (i.e., a graph without loops
and multiple edges) embedded in $F^2$ so that each face is bounded by a 3-cycle and any two faces
share at most one edge. In other words, the vertices, edges and
 faces of $G$ (the corresponding sets denoted by $V(G)$, $E(G)$ and $F(G)$, respectively)
 form a simplicial complex whose underlying space is $F^2$. Two triangulations  $G$ and $G'$ of $F^2$
 are {\it equivalent} if there is a homeomorphism $\varphi: F^2 \to F^2$ with $\varphi(G)=G'$.
In this paper surfaces are supposed to be compact and connected and possibly with boundary.
Surfaces without boundary will be termed {\it closed surfaces}. Here, we distinguish between
triangulations only up to equivalence.

Generation from irreducible triangulations of a surface $F^2$ is a well-known procedure for
obtaining all the triangulations of $F^2$. Recall that an edge of a triangulation $G$ of $F^2$  is
{\it contractible} if the vertices of the edge can be identified (multiple edges are removed, if
they appear) and the result is still  a triangulation of  $F^2$  (\cite{Barnette2}). A
triangulation is said to be {\it irreducible} if it has no contractible edge. Irreducible
triangulations form a generating set for all triangulations of the same surface in the sense that
every triangulation of the surface can be obtained from some irreducible triangulation by a
sequence of vertex splittings (the inverse of the edge contraction operation); see
\cite{Barnette2}.

Barnette and Edelson \cite{Barnette2} showed that every closed surface has finitely many
irreducible triangulations.    More recently, Boulch, Colin de Verdi\`{e}re, and Nakamoto
\cite{Nakamoto2} showed the same result for compact surfaces with a nonempty boundary.
Notwithstanding, it is far from being trivial to enumerate the irreducible triangulations of a
given surface. Complete lists of irreducible triangulations are available only for some low genus
surfaces. See \cite{Negami2}, \cite{S2006-2}, \cite{S2006-3} for a comprehensive reference for the
class of closed surfaces.

So far, the research on irreducible triangulations of closed surfaces has produced a considerable
literature. This is not the case for surfaces with boundary, for which few references can be
presently found; see \cite{Nakamoto2}, \cite{CLQV2014}. This paper is a contribution to the study
of irreducible triangulations for {\it punctured surfaces} (i.e., surfaces with a hole produced by
the deletion of the interior of  a disk in closed surfaces).

It is well known  that any irreducible triangulation $G$ of an arbitrary non-spherical closed
surface $F^2$ has minimum degree
 $\geq 4$ \cite{Negami2}. This is no longer
true if $F^2$ has non-empty boundary. However, if a boundary vertex $v$ has degree $2$ then $v$
lies in exactly one face of $G$ whose boundary edges are trivially contractible (unless $G$ reduces
to a triangle). Therefore we will deal exclusively with triangulations
 with minimum degree $ \geq 3$. Notice that this
condition  implies that no face  shares more than one edge with the boundary of the surface.  This
way, all irreducible triangulations of a surface with boundary $F^2$ (other than the disk) are
elements of the class  $\mathcal{F}_{\circ}^2(4)$ consisting of all triangulations  of $F^2$   with
minimum degree $\geq 3$ and $deg(v)\geq 4$ for all vertices $v$ missing the boundary.

In this paper we give a generating theorem for such triangulations in terms of  internal operations
in the class $\mathcal{F}_{\circ}^2(4)$; that is, we show that all triangulations of a  punctured
surface  other than the disk can be reduced to an irreducible triangulation by performing such
operations   (Theorem \ref{inner4}). The particular case of the disk is also treated (Theorem
\ref{disk}).

 Similarly, we introduce a set  of internal operations in the subfamily
$\mathcal{F}^2(4)\subseteq \mathcal{F}_{\circ}^2(4)$ consisting of
 all triangulations with minimum degree $\geq 4$. In contrast  with the case
  of closed surfaces, for a surface with non empty boundary, the minimal
triangulations obtained by the use of  such operations may contain contractible edges whose
contraction produce 3-valent vertices. We prove  that  such contractible edges are necessarily
located in two particular configurations given in Definitions \ref{defquasiOctahedron} and
\ref{de:Mconfiguration}, see Theorem \ref{teor3}.

 The main  results collected in this work can be regarded as extensions  to punctured surfaces of the main
theorems by Nakamoto and Negami for closed surfaces in \cite{Nakamoto}.

\section{Notation and preliminaries}
If $G$ is a triangulation of the surface $F^2$,  let $\partial G \subset G$  denote  the subgraph
triangulating  the boundary $\partial F^2$. The vertices and edges of $\partial G$ will be called
{\it boundary vertices} and {\it boundary edges} of $G$, respectively. The vertices and edges of
$G-\partial G$ will be called {\it inner vertices} and {\it inner edges} of $G$, respectively. Let
us now recall that the link of a vertex $x \in G$, denoted $link(x)$, is the set of edges in $G$
which   jointly with the vertex $x$ form a triangle in $G$.

Let $e=v_1v_2$ be an edge in $G$. Let us recall that the distance from $e$ to  $\partial G$,
denoted $d(e, \partial G)$, is defined to  be the minimum number of edges needed  to connect $e$
and $\partial G$.
The resulting graph obtained by contracting $e$ in $G$ is denoted by $G/e$. The contraction of a
pair of disjoint edges in two adjacent faces in $G$ is named {\it double contraction}. If
$e=v_1v_2$ is a contractible edge of $G$, then the new vertex $v=v_1=v_2$ in $G/e$ satisfies
$deg(v)=deg(v_1)+deg(v_2)-3$ when $e$ is a boundary edge of $G$, and $deg(v)=deg(v_1)+deg(v_2)-4$
otherwise. Here $deg(v)$ denotes the degree of  the vertex $v$, if $deg(v)=k$ we say that $v$ is a
{\it $k$-valent} vertex. A 3-cycle in $G$ is {\it critical}  if it consists of three edges which do
not bound a face of $G$. Besides, if $xv_1v_2$ is a face of $G$, then $deg(x)$ diminishes by one
after the contraction of $e$.

When a lower bound $k$ for the degree of the vertices is preserved after  contractions  we will use
the
 term {\it $k$-contraction}. Namely, given a triangulation $G$  with minimum degree $\geq k$, an
edge $e$ is said to be {\it $k$-contractible} ({\it $kc$-edge} for short)  if  the minimum degree
of  $G/e$ is at least $k$.  If an edge $e$ is contractible but not $k$-contractible, we call $e$ to
be a {\it $cnkc$-edge}, for short.  The inverse operation of the contraction of the edge $e=v_1v_2$
is the {\it splitting} of $v_1=v_2$. When $deg(v_i)\geq k$ for $i=1,2$, after the splitting, this
will be called a {\it $k$-splitting}.

\begin{remark}{\rm
Notice that the contraction of an edge $e\in G$ belonging to a critical 3-cycle produces a double
edge. On the other hand, if $e$ belongs to  no critical 3-cycle and at most one of its end vertices
belongs to $\partial G$, then $e$ is contractible. In other  words, the impediments  to the
contractibility of $e$ are the two following locations of $e$ in $G$:
\begin{itemize}
\item[(1)] $e$  belongs to a critical cycle of $G$. This is  the case if $e$ lies on the boundary  of a hole of length $3$.
\item[(2)] $e$ is an inner edge but its two vertices belong to  $\partial G$.
\end{itemize}

}
\end{remark}

\begin{remark}\label{re:degree}{\rm
Notice that a  necessary  condition for an interior edge $e=v_1 v_2$ to be $4$-contractible is that
$e$ belongs to faces $v_1v_2v_3$ and $v_1v_2v_4$ so that $\deg(v_i)\geq 5$ for $i=3,4$. If $e$ is a
boundary edge lying in a face $v_1v_2v_3$, the necessary condition for $e$ to be 4-contractible is
that  $\deg(v_3)\geq 5$. }
\end{remark}

The following definitions  extend to  surfaces with boundary the one given in \cite{Nakamoto} for
closed surfaces.

\begin{definition}\label{defOctahedron}{\rm
Let $G$ be a triangulation of a surface  $F^2$ possibly with non-empty boundary. Let $v_1v_2v_3$ be
a cycle  of $G$ such that $deg(v_i)=4$ for $i=1,\,2,\,3$ and
 $\{a_1, a_2, a_3\}$  be the only three vertices
such that $a_i$ is adjacent to $v_j$ and $v_k$ for $\{i, j, k\}=\{1,2,3\}$. The subgraph $H\subset
G$ induced by the vertex set $\{a_1, a_2, a_3, v_1, v_2, v_3\}$ is said to be  an
 \textit{octahedron component}  centered at
$v_1v_2v_3$ with remaining vertices $a_1, a_2, a_3$ if the cycle $a_1a_2a_3$ exists in $G$ (and
hence in $H$) and one of the following conditions holds:

\begin{enumerate}
\item $v_i\notin \partial G$ for all $i=1,2,3$.
\item \label{borde} Only one $v_i$ lies in $\partial G$ and $\partial G$ coincides with
$v_ia_ja_k$.
\item Exactly $v_i, v_j\in \partial G$ and hence  $\partial
G=v_iv_ja_k.$
\item $v_i\in \partial G$  for all $i=1,2,3$ and hence
$\partial G= v_iv_jv_k.$
\end{enumerate}
An octahedron component of $G$ is said to be {\it external} if two edges $a_ia_j$, $a_ja_k$ lie in
$\partial G$ (in particular, $\delta(a_j)=4$). Observe that this happens only under  condition 1. }
\end{definition}
\begin{definition}\label{defquasiOctahedron}{\rm The subgraph $H$ in Definition \ref{defOctahedron} will be termed a
 {\it quasi-octahedron component} of $G$ centered
at $v_1v_2v_3$ and remaining vertices $a_1, a_2, a_3$
if one of the following conditions holds. (Figure \ref{Octaedro}).

\begin{enumerate}
\item The cycle
$a_1a_2a_3$ exists  but does not define a face in $G$ and only one vertex $v_i$  belongs to
$\partial G$ but  (in contrast with \ref{borde} above) $a_ja_k\notin \partial G$.
\item Only the edge $a_ia_j$  fails in closing the cycle $a_1a_2a_3$ in
$G$, hence $v_k\in \partial G$ and all the 3-cycles of $H$ are faces of $G$.

\end{enumerate}
}\end{definition}

\begin{figure}
\centering
\includegraphics[width=0.8\linewidth]{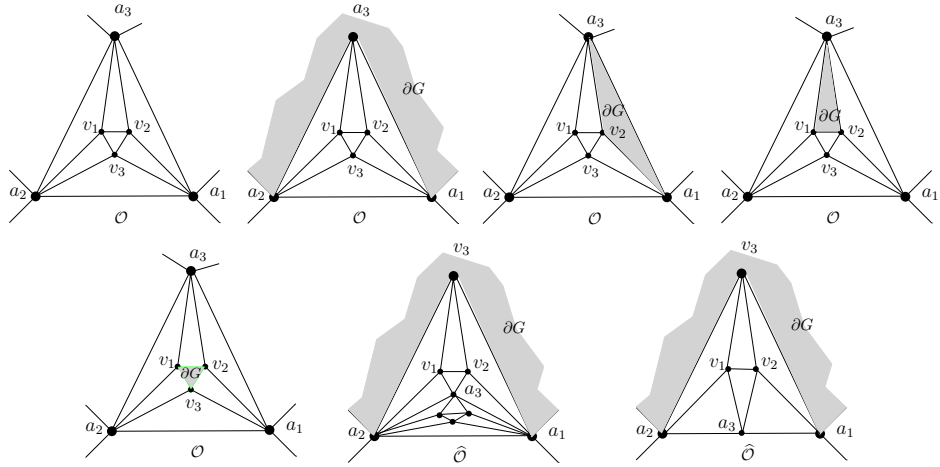}
\caption{Octahedron and quasi-octahedron components in $G$.}
    \label{Octaedro}
\end{figure}

\begin{remark}\label{re:octa}{\rm Notice that the possible occurrences of the subgraph $H$ other than the ones
considered in Definitions \ref{defOctahedron} and \ref{defquasiOctahedron} appear when at least two
edges $a_ia_j$ do not exist in $G$ (and then $v_k\in \partial G$).  If no edge $a_ia_j$ exists in
$G$, then $H=G$ is a triangulation of the disk. On the other hand, if only one edge $a_ia_j$ is in
$G$, then $a_k$ has degree 2 in $\partial G$. }
\end{remark}

\noindent {\bf Notation:} Octahedron  and quasi-octahedron components will be denoted $\mathcal{O}$
and $\widehat{\mathcal{O}}$, respectively.

\begin{remark}\label{re:octa2}{\rm
Let us remark also that at most one 3-cycle of an octahedron component of $G$ may not be a face of
$G$. In such  case, $\partial G$ reduces  to that 3-cycle.

 Let us note that for a
quasi-octahedron component of $G$ the edge $a_iv_3$ is always a  boundary edge of $G$, for $i=1,2$.

Notice that no quasi-octahedron component $\widehat{\mathcal{O}}$ can be extended to an octahedron
component. Indeed, if $v_i\in \widehat{\mathcal{O}}\cap \partial G$, then the 3-cycle $a_1a_2a_3$
is not a face even though the edge $a_ja_k$ opposite to $v_i$ exists (and it is necessarily  an
inner edge). }
\end{remark}

\section{Characterizing the triangulations of  inner degree at least 4.}

It is well known that any irreducible  triangulation of a closed surface other than the sphere has
minimum degree $\geq 4$.
 Therefore irreducible triangulations of punctured surfaces  $F^2$ (other than the disk) must have minimum inner degree $\geq 4$
 (that is, only  boundary vertices are allowed to have degree 3);  that is, they are in the class $ \mathcal{F}_{\circ}^2(4)$  defined above.

 In this section we give a method to construct all triangulations   in $ \mathcal{F}_{\circ}^2(4)$  from irreducible ones by operations which
 keeps all triangulations within this class (Theorem \ref{inner4}). The special case of the disk is also considered (Theorem \ref{disk}).
 This way  we generalize Theorems 1 and 2 in
 \cite{Nakamoto}. The method in \cite{Nakamoto}  is based on the use of 4-splitting and adding octahedra. The existence of 3-valent vertices in the boundary requires two further operations: adding flags and  triode 3-splittings.

Throughout  this section $F^2$ will denote a surface  with connected (possibly empty) boundary.
Recall that $G\in \mathcal{F}_{\circ}^2(4)$ denotes an arbitrary but fixed triangulation of  $F^2$
with all its inner vertices of degree $\geq 4$.

Let us start by  fixing some notation.

\noindent {\bf Notation:} If $x$ is a vertex of $G$ with $deg(x)=4$  we fix notation by calling
$x_1,\, x_2,\, a, \, b$ its neighbours and for the sake of simplicity
 $link(x)$ is written $link(x)=x_1abx_2x_2x_1$ if $x\notin \partial G$  or
 $link(x)=x_1abx_2$ if $x\in \partial G$. This notation will be used throughout this paper without any  further comment.

\begin{figure}[h]
    \begin{center}
      \includegraphics[width=0.6\linewidth]{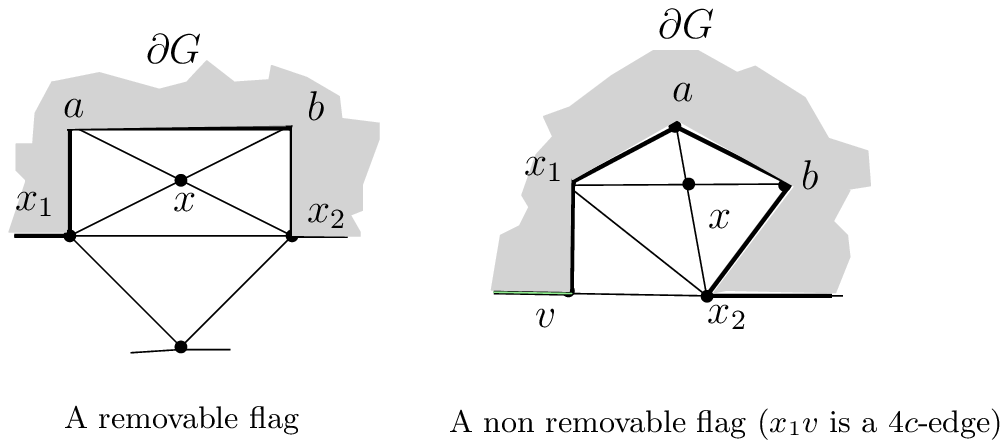}
        \vspace*{.6cm}

        \includegraphics[width=0.7\linewidth]{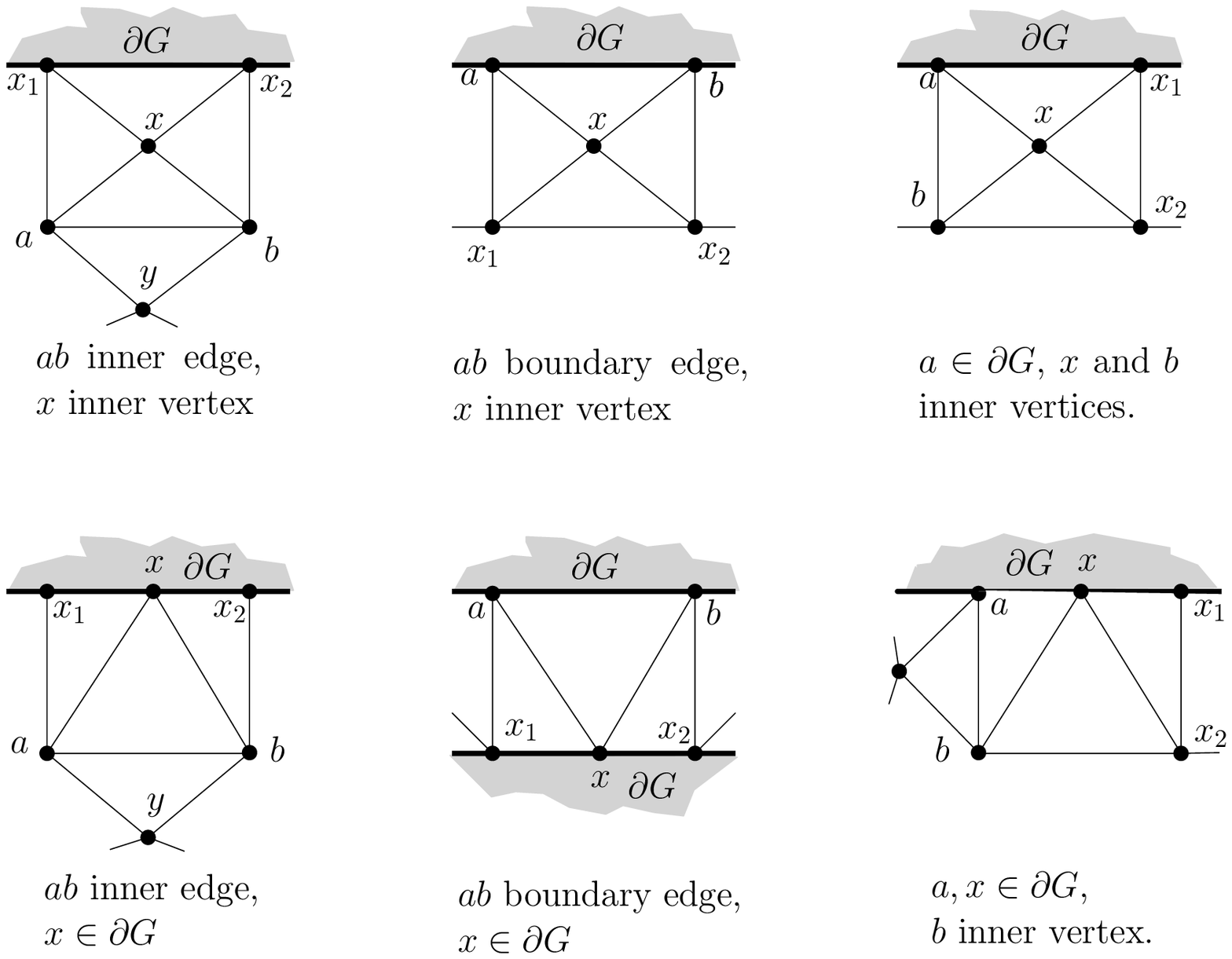}
    \end{center}
    \caption{Different configurations for $link(x)$, with $deg(x)=4$ and distance $\leq 1$ from $\partial G$.}\label{flag}
\end{figure}

The vertices of degree 3 in $\partial G$ play a crucial role in the family
$\mathcal{F}_{\circ}^2(4)$. We will  give them a special name.

\begin{definition}\label{def:triode} {\rm Given  $G \in \mathcal{F}_{\circ}^2(4),$ a boundary vertex  of degree
        3  is called a \textit{ triode } of $G$.
        A contractible edge   of
        $G$ is said to be  a \textit{triode detecting edge} if the vertices  of degree 3 produced by its contraction are triodes.
    }\end{definition}

    \begin{remark}\label{re:triode}{\rm  For every  face $abx$  such that $ab$ is  a
         contractible  boundary edge, $x$  lies in the boundary and $degh(x)=4$ it readily follows
            from Definition \ref{def:triode} that $ab$ is a
            triode detecting edge.

            On the other hand, the contraction of any inner triode detecting edge
            produces at most three triodes since any contraction modifies the degree of at most three vertices.

            It is also  readily checked that two adjacent triodes  define a contractible edge in $\partial
            G$, say  $ab$,
            unless $G$ is isomorphic to the complete graph $K_4$ (and so $G$  triangulates the disk). If, in
            addition, $ab$ shares a face with a 4-valent inner vertex, the contraction of any  edge incident at $x$ or in $link(x)$
            is allowed in $\mathcal{F}_{\circ}^2(4)$, but $ab$ is a $cn4c-$edge. To get rid of this obstacle, we define the following configuration termed
    flag. Recall that a vertex $v$ is said to be independent of degree $k$ if all neighbors of $v$ have
    degree $\neq k$.}
\end{remark}
    \begin{definition}\label{de:flag}{\rm  Given  $G \in \mathcal{F}_{\circ}^2(4),$ let $x$ be an independent
            inner vertex of degree 4 such that
            $link(x)=x_1abx_2x_1$ verifies  $\{x_1a, ab, bx_2\}\subset
            \partial G$, $x_1x_2\cap \partial G =\{x_1, x_2\},$
            and $deg(a)=deg(b)=3$. The subgraph induced by $\{x, \, x_1,\,
            x_2,\,a,\,b\} $
            is called a \textit{ flag centered at $x$}.
            If  the graph $G'=G-\{a, \, b, \,x \}$ remains in $\mathcal{F}_{\circ}^2(4)$,
            the flag is said to be \textit{removable} (see Figure \ref{flag}). Conversely, we say that $G$ is obtained from
            $G'$ by \textit{adding a flag } along a boundary edge of $G'$.}
    \end{definition}

    \begin{remark}\label{removflag}{\rm
            Observe that any flag is removable unless $deg(x_1)=4$ (or $deg(x_2)=4$) and this is the only
            impediment for a flag to being removable. If a flag is non-removable then either $x_1 v$ or $x_2 v$
            is a boundary $4c$-edge.}
    \end{remark}

    The next lemma follows immediately from  definitions  and it will be used in the proof of Theorem \ref{inner4} below.

    \begin{lemma}\label{triode}
        Let $ab$ be a contractible inner edge of $G\in
        \mathcal{F}_{\circ}^2(4),$ let $x$ and  $y$  be vertices so that
     $x$ is a boundary  vertex with $deg(x)= 4$ and $abx$ and $aby$ define two faces of $G$. Then $ab$ is a triode detecting edge
        whenever  $y$ is a boundary vertex of $deg(y)\geq 4$ or  else $y$ is
        an inner vertex of degree $deg(y)\geq 5$.

    \end{lemma}

 Next definitions
introduce the family of removable octahedron components, which added to 3-contractions and
4-contractions of edges lead to a minimal class of irreducible triangulations for any  punctured
surface in the spirit of Nakamoto and Negami's theorem in \cite{Nakamoto}. Recall the notation in
Definition \ref{defOctahedron}.

\begin{definition}\label{addingocta}{\rm
We will say that an octahedron component $\mathcal{O}$ in a triangulation $G \in
\mathcal{F}_{\circ}^2(4)$ is \textit{removable} in $\mathcal{F}_{\circ}^2(4)$ if the graph
$G'=G-\{v_1,\, v_2, \, v_3 \}$ remains in $\mathcal{F}_{\circ}^2(4)$. We also say that $G'$ is
obtained by \textit{removing }  the octahedron $\mathcal{O}$ from $G$. Conversely, $G$ is obtained
from $G'$ by  \textit{adding an octahedron}. }
\end{definition}

\begin{remark}\label{re:cn4cOcta}{\rm

 If $G \in \mathcal{F}_{\circ}^2(4)$
has an octahedron component $\mathcal{O}$, then no  edge  of $\mathcal{O}$ is 4-contractible (see
Remark \ref{re:degree}). However, removing the inner set of vertices $\{v_1,\,v_2,\,v_3\}$ is
equivalent to three consecutive edge 3-contractions ($v_1a_2$, $v_2a_3$ and  $v_3a_1$, for
instance). Therefore, we can regard  this set of 3-contractions as a single operation within the
class $\mathcal{F}_{\circ}^2(4)$ except in case that $\mathcal{O}$ is external. }
\end{remark}

From Definition \ref{addingocta}, the following result gives us sufficient conditions for an
octahedron being removable.

\begin{remark}\label{re:removocta}{\rm
 An octahedron component $\mathcal{O}$ is removable in $G \in \mathcal{F}_{\circ}^2(4)$
 when any of the following cases holds:
 \begin{itemize}
\item All  vertices  $a_1,\, a_2,\,
a_3$,  have degree $\geq 6$.

\item At least one of the vertices $\{a_1,\, a_2,\,
a_3\}$  lies in $\partial G$ and its degree is equal to 5.   Observe that at least one boundary
vertex of degree 3 appears after the removal of $\mathcal{O}$.
 \end{itemize}

 On the other hand, if $\mathcal{O}$  does not hit  $\partial G$ and $deg(a_i)=5$ for
 some $i\in\{1,2, 3\}$, then $\mathcal{O}$ is not removable but the
 edge $a_iv$ is  $4c$-edge where $v$ is the only neighbour of $a_i$
 outside $\mathcal{O}$.
}\end{remark}

Observe that external octahedra are configurations with many  vertices and edges which are not
relevant from the topological point of view. We let them to  be deleted according to the following
definition.

\begin{definition}\label{defaddingOcta}{\rm
 Let $\mathcal{O}$ be an external octahedron component in a triangulation $G \in \mathcal{F}_{\circ}^2(4).$
If  the graph $G'=G-\{a_j, \, v_i, \, i=1,2,3 \}$ remains in $\mathcal{F}_{\circ}^2(4)$,
$\mathcal{O}$ is said to be \textit{redundant} (see Figure \ref{DefREMOVOctaedro}). Conversely, we
say that  $G$ is obtained from $G'$ by  \textit{adding an octahedron } along a boundary edge of
$G'$.

\begin{figure}[h]
\begin{center}
 \includegraphics[width =0.9\linewidth]{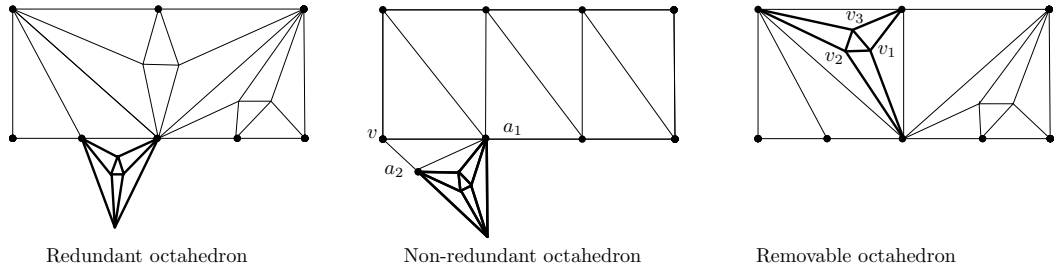}
\end{center}
\caption{Different types of octahedra in triangulations of the M\"{o}bius strip. Here the surface
is represented by a rectangular unfolding with the opposite vertical sides identified in the usual
way. }\label{DefREMOVOctaedro}
\end{figure}
}
\end{definition}

Notice that after deleting  any redundant octahedron component, the new triangulation; that is
$G'=G-\{a_j, \, v_i, \, i=1,2,3 \}$, remains in  $\mathcal{F}_{\circ}^2(4)$. Alternatively  one can
regard the deletion of a redundant octahedron component $\mathcal{O}$ as the composite of the
folding of $\mathcal{O}$ into  a face (as in Definition \ref{folding} below) and the removal  of
the folded octahedron component according Definition \ref{addingocta}

\begin{remark}\label{Config_0}{\rm
Any addition of an octahedron in Definitions \ref{addingocta} and \ref{defaddingOcta} is equivalent
to apply three consecutive splittings in an appropriate set of vertices.

Thus, if  $G$ is a triangulation of any surface $F^2$
 containing  an octahedron
component $\mathcal{O}$, then $G$ is reducible. It suffices to check that the interior edges
$v_iv_j$ and $a_iv_j$ of $\mathcal{O}$ in Figure \ref{Octaedro} are contractible.

 Notice that by  contracting the  three edges $v_iv_j$ of
the triangulation on the right-hand side of Figure \ref{DefREMOVOctaedro},
  we obtain  another triangulation of the M\"{o}bius
strip, depicted with thin lines. }\end{remark}

\begin{remark}\label{nonredundant}{\rm  An external octahedron $\mathcal{O}$ is not redundant
whenever $deg(a_j)=5$ for some $j\in \{1, 2, 3\}$ and this is the only impediment to being
redundant. This provides a triangle $va_ja_k$  with a boundary $4c$-edge $va_j$ (see Figure
\ref{DefREMOVOctaedro}). After contracting $va_j$ two possible situations appear:
\begin{enumerate}
\item[(i)] $\mathcal{O}$  becomes redundant.

\item[(ii)] $\mathcal{O}$ remains non-redundant.
\end{enumerate}

 In case \textit{(ii)}, a new boundary $4c$-edge $v'a_j$ appears
whose contraction leads us again to case \textit{(i)} or \textit{(ii)}. By iterating this
procedure, $\mathcal{O}$ reaches situation $(i)$  in finitely many steps. Otherwise $G$ reduces to
$\mathcal{O}$,   and so $G$ triangulates the disk. }\end{remark}

Next we present the main theorems of this section.

\begin{theorem}\label{inner4} Every triangulation $G\in \mathcal{F}_{\circ}^2(4),$ of a punctured
surface $F^2$, except the disk,  can be obtained from an irreducible triangulation of $F^2$ by a
sequence of 3-splitting triodes, additions of flags, 4-splittings and additions of octahedra.
\end{theorem}

\begin{theorem}\label{disk} Every triangulation $G\in \mathcal{F}_{\circ}^2(4)$ of the disk  can be obtained from a flag or an octahedron component by a
sequence of 3-splitting triodes, additions of flags, 4-splittings and additions of octahedra.
\end{theorem}

The proofs of these results are consequence of the following technical lemma which deals with the
possible configurations near the boundary. This is the crucial difference with the ordinary case of
closed surfaces studied in \cite{Nakamoto}.

\begin{lemma}\label{Config_triode_flag2} Assume that $ab$ is a $cn4c$-edge in $G$, that is,
there is a face $abx$ in $G$ with $deg(x)\leq 4$.
\begin{enumerate}
\item  Let $G\in \mathcal{F}_{\circ}^2(4)$ be a  triangulation of a punctured surface
 different from the disk. If $d(ab, \partial G)\leq 1$, then either
 a $4c$-edge or a subgraph $H\subseteq G$ in
the family
$$\mathcal{A}=\{\mbox{octahedron component}, \,\, \mbox{triode
detecting edge}, \,\, \mbox{flag}\}
$$ can be found at distance at most 1 from $ab$.
\item If $G$ triangulates  the disk, then the subgraph  $G'$  may reduce to a flag or an octahedron.
\end{enumerate}
\end{lemma}

For the sake of simplicity we will give the proof of Lemma \ref{Config_triode_flag2} in the final
appendix.

\begin{remark} {\rm It is straightforwardly checked that an octahedron component $\mathcal{O}$ is not contractible to  a
flag within the family $\mathcal{F}_{\circ}^2(4)$; that is, any  contraction of any  inner edge in
$\mathcal{O}$ produces an inner 3-valent vertex. Hence,  flags and octahedra are needed for
generating all triangulations with minimum inner degree 4.}
\end{remark}

\noindent{\bf Proof of Theorems  \ref{inner4} and \ref{disk}.} We will show that for any reducible
triangulation $G\in \mathcal{F}_{\circ}^2(4),$ every $cn4c$-edge which is not a triode detecting
edge lies in a removable octahedron or  in a removable flag. This way, an irreducible triangulation
$G'$ can be obtained recursively from $G$. Conversely, $G$ is constructed from $G'$ by  a sequence
of 3-splitting triodes, additions of flags, 4-splittings and additions of octahedra.

Assume $G$ contains neither  $4c$-edges nor triode detecting edge. Since $G$ is reducible, let $ab$
be a $cn4c$-edge in $G$ and therefore, a vertex $x$ with $deg(x) \leq 4$ defines a face $abx$ of
$G$ and $x_1abx_2\subseteq link(x)$ (the edge $x_1x_2$ may exist or not). Let us remark that the
case $d(ab, \partial G) \geq 2$ admits the same kind of arguments given in Lemma 1 of
\cite{Nakamoto} for closed surfaces to find either an octahedron or a  $4c$-edge. Hence, we focuss
on the case $d(ab,
\partial G)\leq 1$. In that case Lemma \ref{Config_triode_flag2} leads us to one of the following
cases:

\begin{enumerate}
\item[(a)] There exists a flag $X$ such that  $d(ab,X)\leq 1$, therefore $X$ is removable  (otherwise a $4c$-edge exists according to
Remark \ref{removflag}).
\item[(b)] There exists an octahedron component $\mathcal{O}$ such that $d(ab,\mathcal{O})\leq 1$, thus $\mathcal{O}$  is removable or redundant in
$\mathcal{F}_{\circ}^2(4)$ (see Remarks \ref{re:removocta} and \ref{nonredundant}).
This finishes the proof.
 \begin{flushright}
 $\Box$
\end{flushright}
\end{enumerate}

\section{On reductions of triangulations of degree at least 4.}
Henceforth, unless otherwise is  stated, by $F^2$ we mean
    any punctured surface.

Recall that $\mathcal{F}^2(4)$ denotes the set of triangulations of the surface $F^2$  with all its
vertices of degree $\geq 4$. In this section we give a series of reductions involving exclusively
triangulations in $\mathcal{F}^2(4)$. The two operations  introduced by Nakamoto and Negami in
\cite{Nakamoto} are among such reductions  and they are the only ones which are defined in absence
of boundary. In particular, the triangulations of closed surfaces which are minimal for such
reductions coincides with  the irreducible  triangulations in \cite{Nakamoto}. In sharp contrast
with the class  of closed surfaces, for a surface with non empty boundary, the minimal
triangulations obtained by such reductions may contain contractible edges whose contraction produce
3-valent vertices. For this case, we prove in Theorem \ref{teor3} that  those possible contractible
edges are located in two particular configurations given in Definitions \ref{quasi-octahedron} and
\ref{de:Mconfiguration} below.

Notice that  Definitions   \ref{addingocta} and \ref{defaddingOcta} restrict to the family
$\mathcal{F}^2(4)$ in the obvious
 way so that \textit{removable} and \textit{redundant} octahedra as well as  {\it removing} and {\it addition}
 of such configurations are defined in $\mathcal{F}^2(4)$. Besides these operations, we introduce new ones
 in Definitions
\ref{folding}, \ref{quasi-octahedron}, \ref{defOcta-to quasi} and \ref{Negami'ssuggests} below.

\begin{definition}\label{folding}{\rm

Let $G\in \mathcal{F}^2(4)$ be a triangulation of the surface $F^2$. Let $\mathcal{O}$ be an
external octahedron of $G$ so that $deg(a_3)=4$ and $deg(a_i)=6$ for $i=1$ or $2$. Let $v$ be a
vertex of $G$ such that $a_1a_2v$ is a face of $G$. By  \textit{folding the octahedron
$\mathcal{O}$ onto the face  $a_1a_2v$}  we mean the removal of  $\mathcal{O}$ followed by  the
 addition of an octahedron to the face $a_1a_2v$ (Figure
 \ref{insert}). The inverse operation is called  \textit{unfolding an
octahedron} with respect to the boundary of $G$.

\begin{figure}[h]
\begin{center}
\hspace*{1cm}        \includegraphics[width =0.6\linewidth]{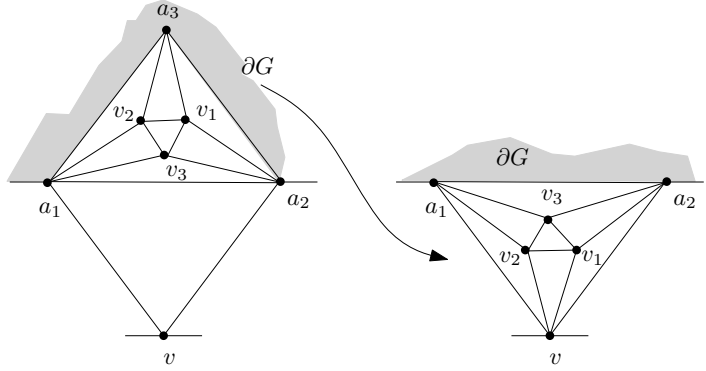}
\end{center}
\caption{Folding the octahedron $\mathcal{O}$ onto the
 face  $a_1a_2v$.}\label{insert}
\end{figure}

}\end{definition}

\begin{definition}\label{quasi-octahedron}{\rm
Let  $G\in \mathcal{F}^2(4)$ be a triangulation of $F^2$.  A {\it quasi-octahedron component} of
$G$, $\widehat{\mathcal{O}}$, is said to be \textit{ removable in} $\mathcal{F}^2(4)$ (or
\textit{4-removable}, for short) if one of the following conditions holds:
\begin{enumerate}
\item The graph $G'=G-\{v_1,\, v_2, \, v_3 \}$ yields a triangulation of $F^2$ in $\mathcal{F}^2(4)$.
\item The graph $G'=(G-\{v_i,\, v_j \})/a_iv_k $, with $v_k\in \partial G$,  yields a
triangulation of $F^2$ in $\mathcal{F}^2(4)$.
\end{enumerate}
 In both cases, we will simply say that $G'$ is obtained by
\textit{removing a quasi-octahedron} from $G$. Conversely, if case (1) happens, we say that $G$ is
obtained from $G'$ by  \textit{adding a quasi-octahedron} along two boundary edges $a_ia_3$ (for
$i=1,\,2$) of $G'$. In case (2) we say that $G$ is obtained from $G'$ by {\it embedding a
quasi-octahedron} in a boundary face $a_1a_2a_3$ of $G'$.
\begin{figure}[h]
\begin{center}
\hspace*{1cm}         \includegraphics[width =0.8\linewidth]{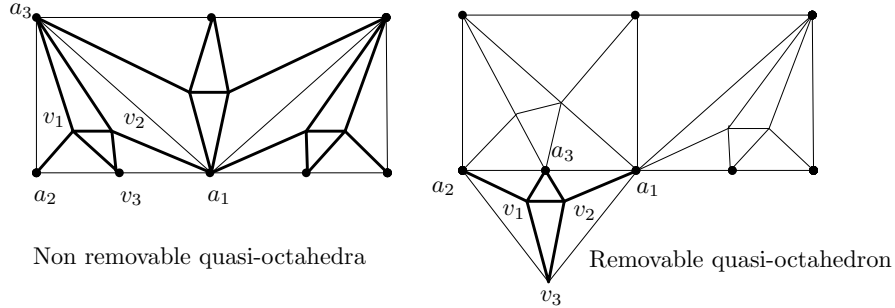}
\end{center}
\caption{ Triangulations for the M\"{o}bius strip with some quasi-octahedra
components.}\label{EjDefCuasioctaedro}
\end{figure}

}\end{definition}

\begin{remark}{\rm
Adding a quasi-octahedron  is  equivalent to apply  three consecutive splittings starting in a
boundary vertex.

As a consequence, if $G$ is a  triangulation of a surface $F^2$
containing   a quasi-octahedron component $\widehat{\mathcal{O}}$, then $G$ is reducible. Indeed,
similarly to Remark \ref{Config_0}, the interior edges $v_iv_j$ and $a_iv_j$ of
$\widehat{\mathcal{O}}$ in Definition \ref{quasi-octahedron} are readily checked to be
$cn4c$-edges. }\end{remark}

\begin{definition}\label{defOcta-to quasi}{\rm Let  $G\in \mathcal{F}^2(4)$  be a triangulation of the surface $F^2$. Let
 $\mathcal{O}$ be an octahedron component of $G$ so that only the edge
$a_1a_2$ lies in the boundary and $deg(a_2)=5$. Let $v$ be the only neighbor of $a_2$ outside
$\mathcal{O}$. The \textit{ replacement of the boundary octahedron} $\mathcal{O}$ by a
quasi-octahedron $\widehat{\mathcal{O}}$ is defined to be the removal of the edge $a_1a_2$ followed
by the contraction  of the edge $a_2v$ in $G$.  (Figure \ref{Octa-to-quasi}). The inverse operation
is called the \textit{ replacement of the quasi-octahedron} $\widehat{\mathcal{O}}$  by a boundary
octahedron $\mathcal{O}$.

\begin{figure}[h]
\begin{center}
\hspace*{1cm}        \includegraphics[width =0.7\linewidth] {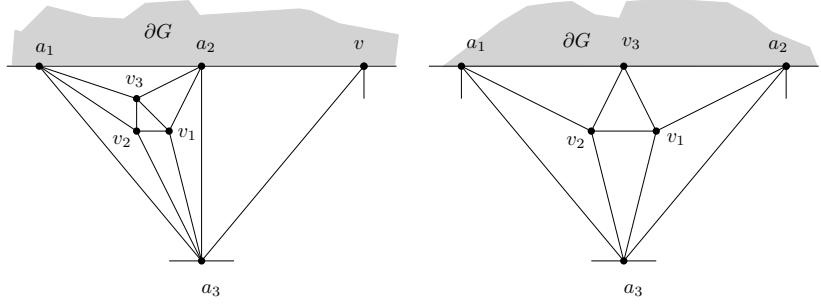}
\end{center}
\caption{A replacement of a boundary octahedron by a quasi-octahedron.}\label{Octa-to-quasi}
\end{figure}
}\end{definition}

\begin{remark}{\rm
Let us observe that the operation in Definition \ref{defOcta-to quasi} is well defined, that is the
edge $a_2 v$ always becomes 4-contractible after the removal of $a_1a_2$.  Indeed, otherwise, there
would be a critical 3-cycle $a_1a_2v$ but no such a cycle exists when the edge $a_1a_2$ is removed.
Notice also that the replacement of a boundary octahedron is needed only in case that the edge
$a_2v$ is not  contractible (or, equivalently, 4-contractible since $deg(a_3) \geq 5$).
 Moreover, if $a_1v\in \partial G$ ($a_1v\notin \partial G$, respectively), then an octahedron (quasi-octahedron, respectively)
 component  would be obtained after applying this reduction.
 In case  $a_1v\notin \partial G$, a quasi-octahedron would be obtained.
}\end{remark}

Observe that  with the previous operations, triangulations with arbitrarily  many  vertices may
appear by repeating the pattern shown in Figure \ref{doublecontraction} and they  could not be
simplified  (or reduced). Next we define a new operation in order to  avoid such an undesirable
repetitive construction.

\begin{definition}\label{Negami'ssuggests}{\rm
Let $G\in \mathcal{F}^2(4)$ be a triangulation of the surface $F^2.$ An \textit{N-component}, of
$G$ consists of a subgraph $\mathcal{N}$ of $G$ determined by two faces sharing an edge, where
their two non-incident edges are contractible and at least one of them lies in $\partial G$ (Figure
\ref{Negami}).
\begin{figure}[h]
\begin{center}
\hspace*{.5cm}        \includegraphics[width =0.3\linewidth] {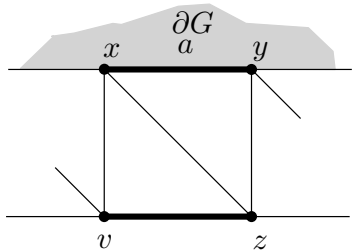}
\end{center}
\caption{$N-$component: two  faces share the  edge $xz$ the non-incident  edges $xy$ and $vz$ must
be contractible and, at least one of them must lie in $\partial G$.}\label{Negami}
\end{figure}
}\end{definition}

\begin{definition}\label{Negami'ssuggests}{\rm Let $G\in \mathcal{F}^2(4)$ be a triangulation of the surface
$F^2$. An $N$-component $\mathcal{N}\subset G$ is termed  \textit{contractible} if  both
contractible edges lie in the boundary or else some inner vertex in $\mathcal{N}$ has degree $\geq
5$. In such  cases the simultaneous contraction of the two non-incident  contractible edges in
$\mathcal{N}$ can be carried out within the family $\mathcal{F}^2(4)$. This reduction will be
called  a \textit{double contraction}.}
\end{definition}
\begin{figure}[h]
\begin{center}
\includegraphics[width=0.8\textwidth] {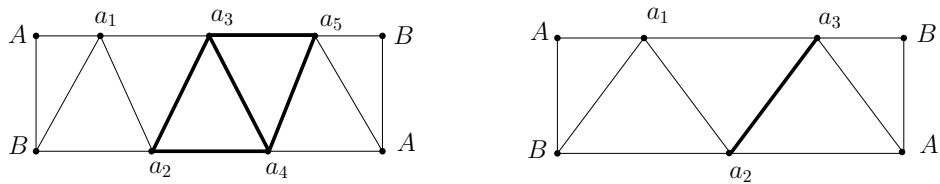}
\end{center}
\caption{The double contraction of the edges $a_2a_4,\, a_3a_5$ in a triangulation of the
M\"{o}bius strip. In this example, both contractible edges lie in $\partial
G$.}\label{doublecontraction}
\end{figure}

\begin{remark}\label{re:doublecont}{\rm
Observe that any double contraction  performed  in an $N$-component with its two inner  vertices of
degree 4 expelled the triangulation from $\mathcal{F}^2(4)$.  In particular, no double contraction
is allowed in any of the two $N$-components sharing an edge in any quasi-octahedron component.}
\end{remark}

The operations defined above and their inverses are summarized in Table \ref{table} where they are
classified into two classes: reductions and expansions.

\begin{table}[h]
\begin{center}
\begin{tabular}{|c l|c l|c|}
\hline
                  &  \textbf{reductions}                  &                      & \textbf{expansions }           & \textbf{Figure }          \\
\hline
\textbf{$R_1$}    &  edge  contraction                    & \textbf{$E_1$ }      & vertex splitting               &                             \\
\hline
\textbf{$R_2$}    & octahedron  removal                   & \textbf{$E_2$ }      & octahedron   addition          & Figure \ref{DefREMOVOctaedro}  \\
\hline
\textbf{$R_3$}    & folding an octahedron                 & \textbf{$E_3$ }      & unfolding an octahedron        & Figure \ref{insert}            \\
\hline
\textbf{$R_4$}    & quasi-octahedron removal              & \textbf{$E_4$ }      & quasi-octahedron addition      & Figure \ref{EjDefCuasioctaedro}  \\
 \hline
 \textbf{$R_5$}   & boundary octahedron replacement       &  \textbf{$E_5$ }     & quasi-octahedron replacement   & Figure \ref{Octa-to-quasi}        \\
\hline
\textbf{$R_6$}    & double contraction in a $N$-component & \textbf{$E_6$}       & double splitting of vertices   & Figure \ref{Negami}       \\
 \hline
\end{tabular}
 \caption{}\label{table}
\end{center}
\end{table}

By a 4-reduction (4-expansion, respectively) we mean any reduction (expansion, respectively) in
Table \ref{table} which provides a new triangulation in  $\mathcal{F}^2(4)$. Notice that reductions
and expansions $R_i$, $E_i$ for $i\geq 3$ are always 4-reductions and 4-expansions.

The operations $R_1$ and $E_1$ preserving $\mathcal{F}^2(4)$ are the usual 4-contractions and
4-splitting of Nakamoto and Negami in \cite{Nakamoto}.

 By the use of these operations we eventually get a class  which is minimal  in
$\mathcal{F}^2(4)$ in the sense of the following definition.

\begin{definition}{\rm
A triangulation   $G\in \mathcal{F}^2(4)$ of the surface $F^2$ is said to be \textit{minimal in}
$\mathcal{F}^2(4)$ (or $4$-\textit{minimal}\footnote{This term appears in \cite{Malnic} with a
different meaning.}, for short) if $G$ does not admit any further $4$-reduction. }\end{definition}

In particular, the $4$-minimal triangulations in $\mathcal{F}^2(4)$ of any non-spherical closed
surface coincide with the usual irreducible ones since operations $R_i$ and $E_i$, for $i\geq 3$
in Table \ref{table} make sense only for boundary surfaces triangulations. This way, Theorem 1 in
\cite{Nakamoto} can be restated  as follows.

\begin{theorem}\label{Nakamoto's}  Any triangulation of a closed surface  in $\mathcal{F}^2(4)$ can be obtained from a
$4$-minimal triangulation  by a sequence of 4-expansions.
\end{theorem}

 Next we focuss our interest on punctured surfaces. Let us start by considering the following
 examples.

 \begin{example}\label{ex:minimal}{\rm
 \begin{enumerate}
 \item[a)] Let $G$ be the triangulation
  of the punctured connected sum
 of three projective planes constructed by identifying the edges with
 equal labels. See Figure \ref{ex:minimal1}.
 Notice that $G$ is $4$-minimal, it contains a  non-removable quasi-octahedron component and a
  $cn4c$-edge, namely $ab$, outside the quasi-octahedron component.
\begin{figure}[h]
\begin{center}
\includegraphics[width=0.8\linewidth]{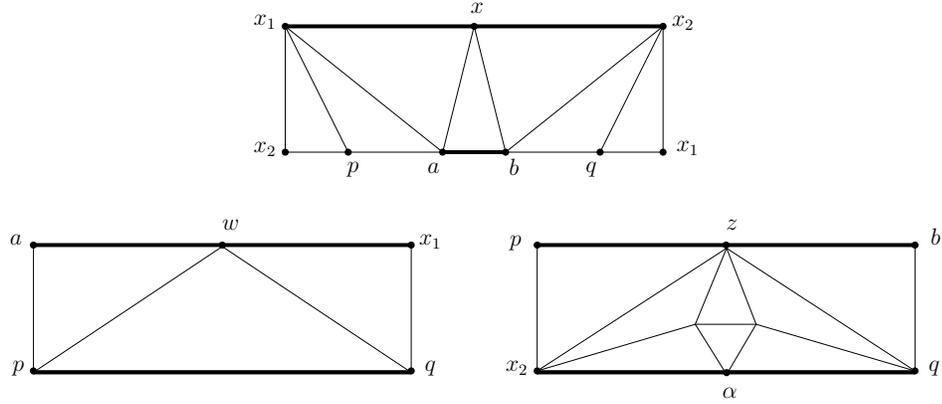}
\end{center}
\caption{Observe that $\delta(a)=\delta(b)=5,$ $\delta(x_1)=\delta(x_2)=6$ and $\partial G$
consists of bold edges. }\label{ex:minimal1}
\end{figure}
\item[b)] Let $G$ be the triangulation of the punctured connected sum of three projective planes with a torus,
  depicted in Figure
 \ref{ex:minimal0}.  The
connected sum is done by firstly identifying the edges with
 equal labels (namely $x_2p,\, pa,\, bq$)   and then by attaching
  two copies of $M_3$ along the boundaries of the polygons labeled   $M_3$ in the
 strip on the right-top.  Notice that $G$ is $4$-minimal and $ab$ is its only  $cn4c$-edge. Here, $M_3$ denotes the so-called irreducible  triangulation of the  M\"{o}bius strip collected in \cite{CLQV2014}.
  \begin{figure}[h]
\begin{center}
 \includegraphics[width =0.8\linewidth]{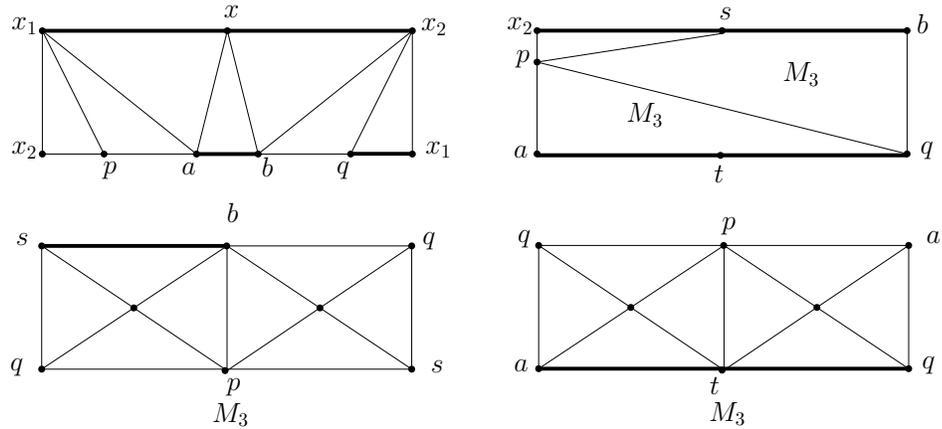}
\end{center}
\caption{Observe that $\delta(a)\geq 6,$ $\delta(b)\geq 6,$ $\delta(x_2)\geq 6,$ $\delta(x_1)=5.$
Here $\partial G= absx_2xx_1qta.$}\label{ex:minimal0}
\end{figure}

\item [c)] Let $G$ be the triangulation of the punctured torus depicted in Figure \ref{toroquasiocta}.
 Notice that $G$ is $4$-minimal and it contains  a unique non-removable quasi-octahedron component.
\begin{figure}[h]
\begin{center}
\hspace*{.5cm}\includegraphics[width=0.3\linewidth]{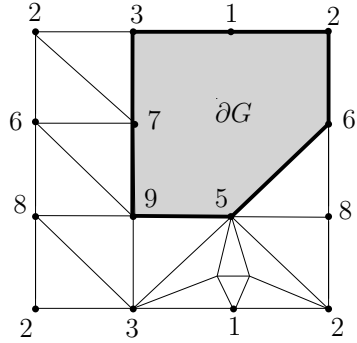}
\end{center}
\caption{A $4$-minimal triangulation of the punctured torus with precisely
 one non-removable quasi-octahedron component. Here, the triangulation is obtained by  removing vertex 4 in $T^{13}$ triangulation
 of the torus given in \cite{Lawrencenko1} and $\partial G=31265973$.}\label{toroquasiocta}
\end{figure}
 \end{enumerate}
}
\end{example}

\begin{remark}{\rm

Examples \ref{ex:minimal} a) and  b) illustrate a procedure to construct $4$-minimal triangulations
by leaving $M$ fixed and by choosing as  $M_3$ any  irreducible triangulation  of a punctured
surface with appropriate boundary length.}
\end{remark}

 Example \ref{ex:minimal} depicts  4-minimal triangulations with contractible  edges. In fact, as we prove in Theorem \ref{teor3} bellow,
    the configurations  of contractible edges in any 4-minimal triangulations are exactly  two: the non-removable quasi-octahedron and the one termed
   $M$-configuration in the following definition.

\begin{definition}\label{de:Mconfiguration}{\rm

Let $G\in \mathcal{F}^2(4)$ be a triangulation of the surface $F^2.$ Let $abx$ be a face with $x$ a
4-valent vertex and $ab$ a boudary $cn4c$-edge.  An \textit{M-component} centered at  $abx$ in $G$
consists of a subgraph $\mathcal{M} \subseteq G$ determined by three faces $\{xab, xax_1, xbx_2\}$
such that
  $xx_1,\, xx_2$ lie in the boundary and
 $x_1x_2$ is an inner edge.
Notice that $xx_1x_2x$ is a critical 3-cycle. (Figure \ref{Mconfiguration}).

\begin{figure}[h]
\begin{center}
\hspace*{.5cm}\includegraphics[width=0.4\textwidth] {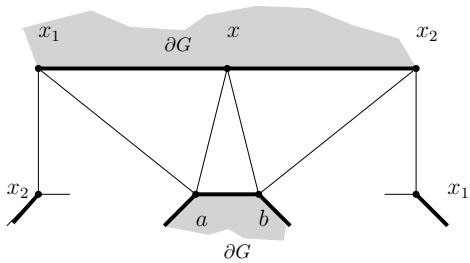}
\end{center}
\caption{$M$-component centered at $abx$. }\label{Mconfiguration}
\end{figure}
}
\end{definition}

\begin{remark}\label{re:M}{\rm In Figure \ref{ex:minimal1} above  both a non-removable quasi-octahedron and an $M$-component appear.
On the other hand, each  Figure \ref{ex:minimal0}, \ref{M_enPuncturedTorus4},
\ref{M_enPuncturedTorus5} only contains one $M$-component. Finally, Figure \ref{toroquasiocta}
contains exactly  a non-removable quasi-octahedron.

}
\end{remark}

\begin{remark}\label{lem:M}{\rm In any $M$-component centered at $abx$ in a triangulation  $G\in\mathcal{F}^2(4),$  the boundary edge $ab$ is triode detecting  and
 $deg(x_i)\geq 5$ for $i=1,2$. Hence there are vertices $p, q$ so that $apx_1$ and $bqx_2$ are faces in $G$. This way,
$\{q, x_2, p, a, x\}\subseteq V(link(x_1))$ and $\{p, x_1, q, b, x\}\subseteq V(link(x_2))$.
Moreover,
\begin{enumerate}
\item If  $deg(a)=4$ (or $deg(b)=4$), then, $ap$ ($bq$ respectively) is a boundary $4c$-edge. M\"{o}bius
\item Otherwise, there are faces $apw$, $bqr$.  If $deg(a)=5$ (or  $deg(b)= 5$)
 then the edge $aw\subset \partial G$ and it is not a triode detecting edge although $aw$ may be contractible.
\end{enumerate}

Let us observe that in case that all 3-cycles $x_1x_2p$, $x_1ap$, $x_1x_2q$ and $x_2bq$ are faces
of the triangulation, the $M$-component centered at $abx$  coincides with the  triangulation
obtained  by the splitting of a 5-valent boundary vertex in the irreducible triangulation of the
strip $M_2$ collected in \cite{CLQV2014}.

On the other hand, it is not difficult to  see that the set of vertices $\{a, b, x_1, x_2, x\}$ in
the $M$-component are principal vertices of a subdivision of the complete graph $K_5$ in $G$.
Hence, $M$ is not present in any triangulation of the disk.
  \begin{figure}[h]
\begin{center}

\hspace*{.5cm}\includegraphics[width=0.7\textwidth]{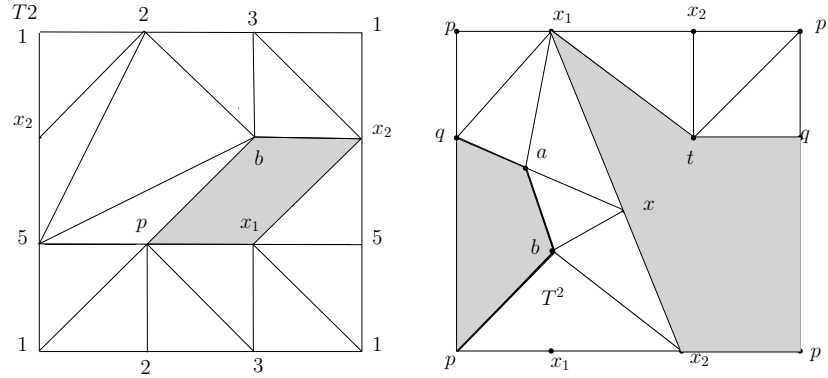}
\end{center}
\caption{A  triangulation of the punctured double torus with precisely
 one M-component centered at $abx$. Here $\partial G= x_1xx_2pbaqtx_1$. The connected sum
 is done by  attaching the punctured torus $T^2$ along the cycle $bpx_1x_2b$ on the right side
 torus.}\label{M_enPuncturedTorus4}
\end{figure}

\begin{figure}[h]
\begin{center}
 \hspace*{.5cm}\includegraphics[width=0.7\textwidth]{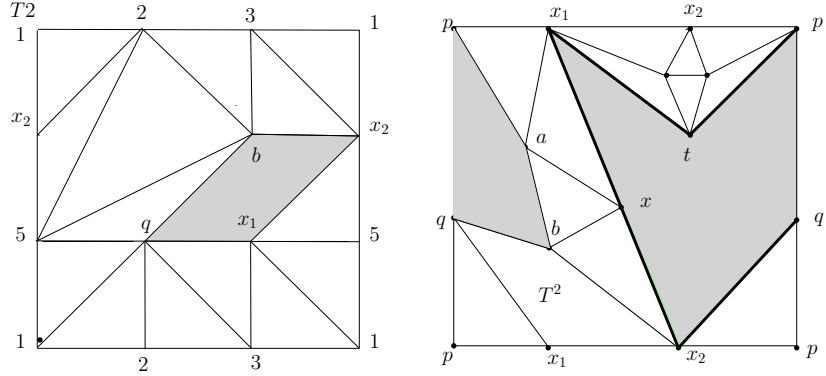}

\end{center}
\caption{A  triangulation of the punctured double torus with precisely
 one M-component centered at $abx$ and one non-removable quasi-octahedron. Here $\partial G= x_1xx_2qbaptx_1$. The construction is done
 in a similar way as in Figure  \ref{M_enPuncturedTorus4}.}\label{M_enPuncturedTorus5}
\end{figure}

  \begin{figure}[h]
\begin{center}
 \includegraphics[width = 0.8\linewidth]{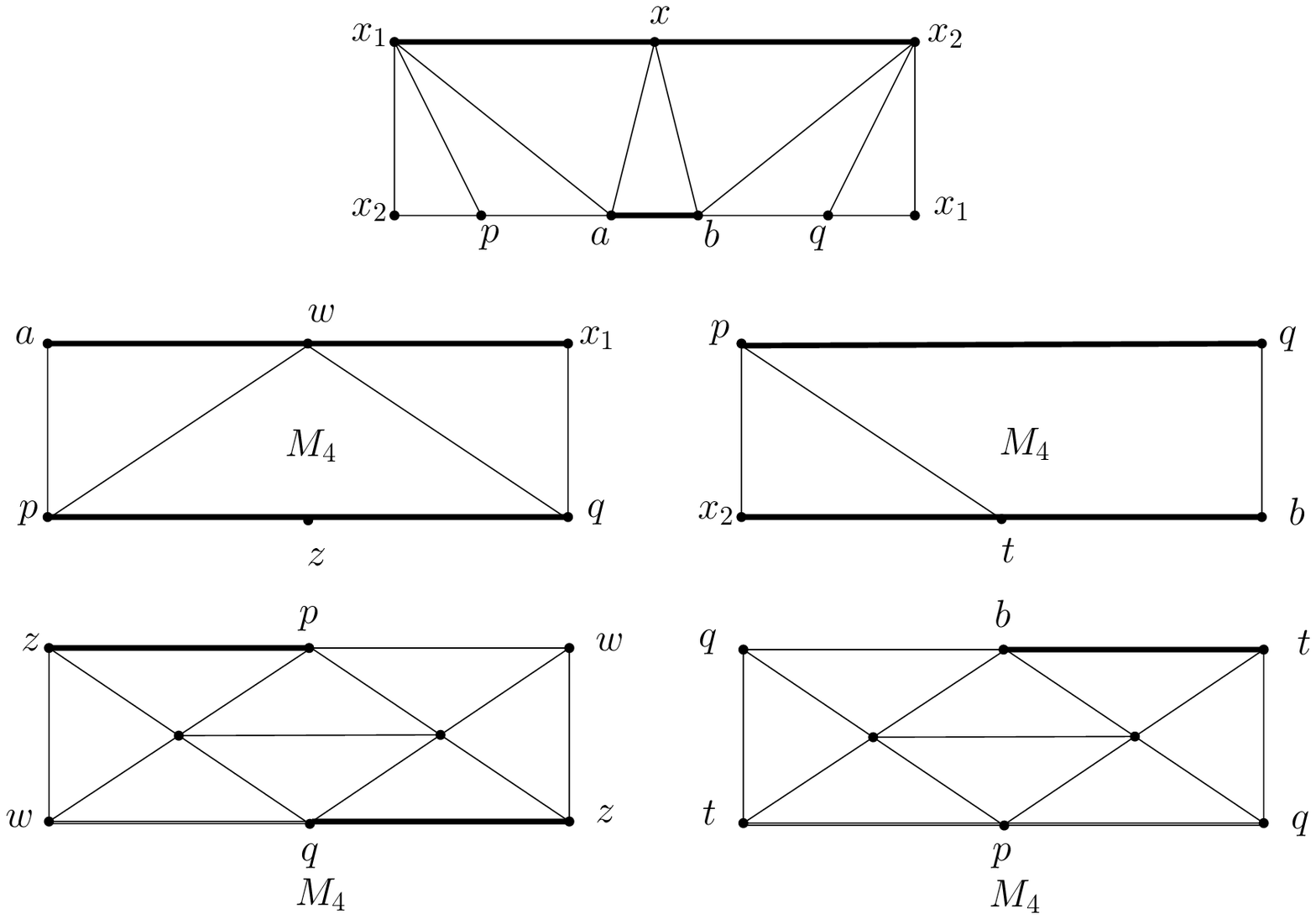}
\end{center}
\caption{Observe that $\delta(a)=5, $ $\delta(b)\geq 6,$ $\delta(x_2)\geq 6, \, \delta(x_1)\geq
6.$}\label{ex:minimal5}
\end{figure}

}

\end{remark}

The interest of the  $M$-component is pointed out by the following proposition.

\begin{proposition}\label{pro:Mstable} Let $G\in  \mathcal{F}^2(4)$ be  a triangulation  of the surface $F^2$,
 different from the disk. Any  $M$-component $\mathcal{M}\subset G$ remains unaltered  after performing any reduction $R_i$ ($i=1, \dots, 6$).
\end{proposition}

\begin{proof}
Let $\mathcal{M}$ be an $M$-component centered at $abx$. The edge $ab$ is the only contractible one
in $\mathcal{M}$ (in fact, it is a $cn4c$-edge), hence no reduction $R_1$ can be applied to
$\mathcal{M}$.

Furthermore, the only  possible octahedron or quasi-octahedron  components containing $ab$ must be
centered at $abx$ (since $deg(x_i)\geq 5$ for $i=1, 2$, by  Remark \ref{lem:M}). However, such a
quasi-octahedron component  can not exist  since $a, b \in \partial G$. Similarly no octahedron
components exists since otherwise $\partial G$ reduces to  $abx$. Therefore, no  reduction $R_2$ to
$R_5$ may be performed  to  $\mathcal{M}$.

Finally, the existence of an $N$-component containing  $ab$ is ruled out by existence of the edge
$x_1x_2$. Then, reduction $R_6$ does not affect $\mathcal{M}$. This finishes the proof.
\end{proof}

We are now ready to establish and prove the main result of this section. Namely,

\begin{theorem}\label{teor3}  A  triangulation $G\in \mathcal{F}^2(4)$
of the surface $F^2$, different from the disk, is $4$-minimal if and only if each contractible edge
(if any) of $G$
 is located in either a non-removable quasi-octahedron
component or an $M$-component.
\end{theorem}

\begin{theorem}\label{disk4} The only  $4$-minimal triangulation
of the  disk is the octahedron.
\end{theorem}

In order to  prove  Theorems \ref{teor3} and \ref{disk4}  we will need  the following technical
lemmas.

\begin{lemma}\label{Config_1} Let $G \in \mathcal{F}^2(4)$ be a  triangulation of the surface $F^2$
 different from the disk. Assume in addition that  $G$ contains a
non-4-removable octahedron component, $\mathcal{O}$, then exactly one of the following statements
hold:
\begin{enumerate}
\item There is at least one $4c$-edge $a_it$ with $a_i\in \mathcal{O}$, $t\notin \mathcal{O}$ and $\mathcal{O}$ turns to be
4-removable after contracting $a_it$.
\item $\partial G\cap V(\mathcal{O})\subseteq \{a_1, a_2, a_3\}$ and this intersection contains at least two vertices  $a_i, a_j$.
\end{enumerate}
\end{lemma}

\begin{proof}
Since  $\mathcal{O}$  is non-removable,  $\mathcal{O}$ has an external vertex, say $a_1$, such that
$deg(a_1)=5$. Let us suppose that $\mathcal{O}$ does not intersect the boundary, then $a_1\notin
\partial G$, $deg(a_2), deg(a_3)\geq 6$. Hence there is an edge $a_1t$ in $G- E(\mathcal{O})$ existing the two
faces $ta_1a_2$ and $ta_1a_3$  in $G$. This way, $a_1t$ is a $4c$-edge. After contracting it,
$deg(a_1)$ increases and the octahedron turns to be 4-removable.

Moreover, if the intersection $V(\mathcal{O}) \cap \partial G$ reduces to a single  vertex,   let
us suppose that this intersection is precisely  the vertex $a_2$. Since no  edge $a_ia_j$ lie in
$\partial G$ then $deg(a_2)\geq 6$ holds and since $\mathcal{O}$ is non-removable, $deg(a_i)=5$ for
$i=1$ or $i=3$. Let us suppose $deg(a_1)=5$ and let $t$ be the boundary vertex adjacent to $a_2$
and $a_1$. Therefore $ta_1a_2$ and $ta_1a_3$ are faces of $G$, and since $a_1$ is not a boundary
vertex,  it readily follows that $a_1t$ is a $4c$-edge of $G$. Again, after contracting it, the
octahedron become 4-removable. Therefore, $\partial G\cap V(\mathcal{O})$ contains at least two
vertices  $a_i, a_j$.
\end{proof}

 In the following lemma we will use the operations  $R_i$ in Table \ref{table}.

\begin{lemma}\label{Conf} Let $G\in \mathcal{F}^2(4)$ be a triangulation of the surface  $F^2$,
 and  let  $\mathcal{O}$ be a non-removable
octahedron component of $G$ with at least two vertices $a_i, a_j$ in $\partial G$. Then
\begin{enumerate}
\item  If $E(\mathcal{O})\cap \partial G =\{a_1a_2\}$ and $deg(a_2)=5$,
then, there is precisely one vertex $v\in link(a_2)-V(\mathcal{O})$ and by applying an operation
$R_5$ to $G$, the new triangulation  $G'$ belongs to $\mathcal{F}^2(4)$.
\item If  $E(\mathcal{O})\cap \partial G =\{a_2a_3, a_1a_3\}$ and $\delta(a_1)=6$ or $\delta(a_2)=6$,   then
by applying an operation $R_3$ to $G$, the new triangulation $G'$ belongs to $\mathcal{F}^2(4)$.
\item If  $E(\mathcal{O})\cap \partial G =\emptyset$, then there exists precisely one  vertex $a_j$ such that $deg(a_j)=5$
and there is a  $4c$-edge incident with $a_j$.
\end{enumerate}
\end{lemma}

\begin{proof} It is straightforwardly deduced from the definitions involved in the statements and Lemma \ref{Config_1} (2).
Observe that $V(\mathcal{O})\cap \partial G =\{a_1, \, a_2, \, a_3\}$ implies that $\mathcal{O}$ is
4-removable.
\end{proof}

As a consequence of Lemmas \ref{Config_1} and \ref{Conf} we get
\begin{corollary}\label{any-octa} Any octahedron component of a  triangulation in $\mathcal{F}^2(4)$  of the surface $F$ can be deleted by one of the
reductions $R_1$, $R_2$, $R_3$ or $R_5$ of Table \ref{table}.
\end{corollary}

The following is the corresponding analogue of Lemma \ref{Config_triode_flag2} for the class
$\mathcal{F}^2(4)$.

\begin{lemma}\label{Config_6}  Let $G\in \mathcal{F}^2(4)$ be a  triangulation of the
surface $F^2$.
 If $ab$ is a $cn4c$-edge
in $G$ so that $d(ab, \partial G)\leq 1$, then one of the following configurations can be found at
distance at most 1 from $ab$:
\begin{enumerate}
\item A $4c$-edge
\item A subgraph in the family

$\mathcal{B}=\{\mbox{octahedron component}, \,\, \mbox{quasi-octahedron component}, \,\,
\mbox{$N-$component}\}$
\item An $M$-component centered at $abx$.
\end{enumerate}
\end{lemma}

The proof of  Lemma \ref{Config_6} is a specialization of the proof of Lemma
\ref{Config_triode_flag2} and it will be postponed to the final appendix.

\vspace*{0.5cm}

{\bf Proof of  Theorems \ref{teor3} and \ref{disk4}:} Let $ab$ a contractible edge in $G$. As $G$
is  $4$-minimal, $ab$ is a $cn4c$-edge. Moreover, if $d(ab,\partial G))\geq2$ then the same
arguments given in Lemma 1 of \cite{Nakamoto} for closed surfaces allows us to find a $4c$-edge or
an octahedron component at distance $\leq 1$ from $ab$. This contradicts the $4$-minimality of $G$.
Thus, necessarily, $d(ab, \partial G)\leq 1$ and Lemma \ref{Config_6}, Corollary \ref{any-octa}
and, again, the $4$-minimality of $G$ yield that $ab$ lies in a non-removable quasi-octahedron
component or an $M$-component.

Conversely,  if the contractible edge $ab$ belongs to a non-removable quasi-octahedron component
$\widehat{\mathcal{O}}$ then it is  not  a $4c$-edge since $\widehat{\mathcal{O}}$ does not contain
such edges by  Remark \ref{re:cn4cOcta}. Moreover, $\widehat{\mathcal{O}}$ cannot be extended to an
octahedron  in $G$ by  Remark \ref{re:octa}. Finally, no  $N$-component  contained in
$\widehat{\mathcal{O}}$ can be reduced by  Remark \ref{re:doublecont}. Hence, no reduction  $R_i$
can be applied to  remove $ab$.

On the other hand, if $ab$ belongs to  an $M$-component $\mathcal{M}\subset G,$ we know by
Proposition   \ref{pro:Mstable} that   $\mathcal{M}$ is stable under reductions $R_i$ ($i=1, \dots,
6$). This finishes the proof of Theorem \ref{teor3}.

Let us consider  the case of the triangulated disk. From Remark \ref{lem:M}  no $M$-component may
appear in a triangulation of the disk. Besides, a quasi-octahedron component
$\widehat{\mathcal{O}}$ will be always removable  according to Definition \ref{quasi-octahedron}.
In fact, it is clear that the degree $\geq 4$  condition expels the quasi-octahedron from the set
of disk triangulations.  Moreover, according to Definition \ref{defquasiOctahedron}, vertex $a_3$
must have degree $\geq 5$. Let $a_3t$ be an edge with $t$ outside $\widehat{\mathcal{O}}$. Observe
that $deg(a_3)=5$   leads to  the  contractibility of $at$, which contradicts the minimality of
$G$, hence $deg(a_3)\geq 6$. Besides, $deg(a_i)\geq 5$ for $i=1,2$ since otherwise a 4-contractible
edge incident at $a_i$ appears, which is impossible. Therefore,
 $\widehat{\mathcal{O}}$ can be removed by applying Definition \ref{quasi-octahedron} (1) if  $deg(a_i)\geq 6$ for $i=1,2$ and
  $a_3\in \partial G$ or  Definition \ref{quasi-octahedron} (2) otherwise. This finishes the proof of Theorem \ref{disk4}.

\begin{flushright}$\Box$
\end{flushright}

\begin{corollary}\label{no-cuasi}Let $G$  be a triangulation of a punctured surface different from the disk such that $G$  is
  $4$-minimal. Then $G$ is irreducible if and only if $G$ contains neither   quasi-octahedron component nor $M$-component.
\end{corollary}

Theorem \ref{teor3} shows that 4-reductions do not suffice to  get all irreducible triangulations
within the class $\mathcal{F}^2(4)$. If, similarly as in \cite{Komuro} for closed surfaces, we
allow diagonal flips then we get the following theorem.

\begin{theorem}\label{last} If diagonal flips are added to 4-reductions as admissible operations in the family
$\mathcal{F}^2(4)$ of triangulations of any punctured surface $F^2$, then the 4-minimal
triangulations reduce to the irreducible triangulations in $\mathcal{F}^2(4)$.
\end{theorem}

\begin{proof} The diagonal flip operation is a way of getting rid of quasi-octahedra and $M$-components in $4$-minimal triangulations.
For instance, if we flip the edge $xa$ in an $M$-configuration when $\deg(a) \geq 5$ (similarly,
flip $xb$ when $\deg(b) \geq 5$) we still have a triangulation in $\mathcal{F}^2(4)$ but now the
edge $ab$ is $4$-contractible. Notice that $deg(x_1) \geq 5$ by definition of an $M$-configuration
and, moreover, that some $4$-contractible edge is detected whenever $\deg(a) = 4$ ($deg(b) =4$,
respectively); see Remark \ref{lem:M}(1).

On the other hand, by  flipping an  edge $a_ia_3$ of a quasi-octahedron component, new
$4$-contractible edges are available to perform further $4$-reductions and dismantle the original
quasi-octahedron component.

This way, any 4-minimal  triangulation turns to be irreducible  within the class
$\mathcal{F}^2(4)$.

\end{proof}

\vspace*{.75cm}
\par
\noindent{ {\large {\bf Appendix: Proofs of Lemmas  \ref{Config_triode_flag2} and
\ref{Config_6}.}}\label{jointII}} \vspace*{.5cm}
\par
\noindent In order to  prove Lemma \ref{Config_triode_flag2}, let $G \in \mathcal{F}_{\circ}^2(4)$
be a fixed triangulation  of the punctured surface $F^2$. Assume that $ab$ is a $cn4c$-edge in $G$,
that is, there is a face $abx$ in $G$ with $deg(x)\leq 4$. We start with the following technical
lemmas which detect possible contractible edges and $4c$-edges around $ab$.

\begin{lemma}\label{previous} Let $ab$ be a  $cn4c$-edge of $G\in\mathcal{F}_{\circ}^2(4) $ so that $deg(a)\geq 4$, $deg(b)\geq 4$,
    and let $x$ be a vertex of degree 4 so that $abx$ defines a face of $G$
    and $V(link(x))=\{x_1, a, b, x_2\}$.  Then, the following statements
    hold:
    \begin{enumerate}
        \item  $ax_2$ and $bx_1$  are not edges of $G$.
        \item Whenever $ab$ and $x$ do not intersect $\partial G$
        simultaneously, $ax$ and $bx$   are contractible edges of $G$.
        \item Whenever $xx_1x_2$ is not a critical 3-cycle, both  $xx_1$ and $xx_2$  are contractible
        edges of $G$.
        \item If $x$ is an inner vertex of
        $G$ and the two vertices $\{a,\,x_2\}$ ($\{b,\,x_1\}$, respectively)
        have degree  $\geq 5$, then the edges $xv$ with $v\in \{b,\,x_1\}$
        ($\{a,\,x_2\}$, respectively), are 4-contractible.
    \end{enumerate}
\end{lemma}
\begin{proof}
    Suppose $ax_2$ (or $bx_1$) is an edge of $G$, since $deg(b)\geq 4$,
    the 3-cycle $ax_2ba$ ($ax_1ba$) does not define a face of $G$, and
    hence $ab$ lies in a critical 3-cycle, contradicting the hypothesis.

    As a consequence, if $ab$ and $x$ do not intersect $\partial G$
    simultaneously, then the edge $ax$ is contractible. Otherwise, a critical
    3-cycle contains $ax$, namely $axta$ with $t\in\{b, x_1, x_2\}$, but
    it is easily deduced that $t=x_2$ is the only possibility, reaching
    a contradiction. A similar argument works for proving the contractibility of
    $bx$. Similarly, $xx_1$ and $xx_2$ are contractible
    whenever $xx_1x_2$ is not a critical 3-cycle.
\end{proof}
\vspace*{.5cm}

\begin{lemma}\label{hits2} If $deg(x)=3$ then there exists a triode detecting
 edge, a flag or a $4c$-edge meeting $link(x)$.
 \end{lemma}

\begin{proof}
Observe that if  $deg(x)=3$, then $x\in \partial G$, $link(x)=x_1ab$ and $a$ or $b$ is an inner
vertex, hence $deg(a)\geq 4$ or $deg(b)\geq 4$. Let us consider $a$ to be the inner vertex, whence
$x_1, b\in\partial G$.

Assume  $deg(a)=4$. Then a flag centered at $a$ appears whenever  $deg(x_1)=3$ or $deg(b)=3$. In
case $deg(x_1)=4$ or $deg(b)= 4$, the edge $xa$ is a triode detecting. This edge turns to be a
$4c$-edge when  $deg(x_1), deg(b)\geq 5$.

On the other hand, if  $deg(a)\geq 5,$  by Lemma \ref{previous} the edge  of vertices $x_1$ and $b$
does not exist.

\end{proof}

In order to  simplify the notation, let $V(link(x))=\{ x_1 , a , b , x_2 \}$ be the vertex set of
the  link of the vertex $x$ fixed in Lemma \ref{previous}. Then the degree of vertices defines a
map $ \delta : V(link(x)) \to \{ n \in \mathbb{N} ; n \geq  3\}$ by $\delta (v) =deg(v)$.

\textbf{Henceforth, let $m=min(\delta)$ denote the minimum of this map, and  $ \sharp  Min$ be the
cardinal of the set $Min=\delta^{-1}(m)$. }

\vspace*{.5cm}
\par
\noindent{\bf Proof of Lemma \ref{Config_triode_flag2}} \vspace*{.5cm}
\par

Since $ab$ is a $cn4c$-edge in $G$,  there is a face $abx$ in $G$ with $deg(x)\leq 4$. The case
$deg(x)=3$ is studied in Lemma \ref{hits2}. For $deg(x)=4$, recall that  we denote $link(x)=
x_1abx_2x_1$ if $link(x)= \{x_1a, ab, bx_2, x_2x_1\}$ and   $link(x)= x_1abx_2$ if $link(x)=
\{x_1a, ab, bx_2\}.$ (See Figure \ref{flag}).

\noindent{\bf Case 0.} Notice that  if $x\in \partial G$ and $ab\subset
\partial G$, then it is clear that $ab$ is a triode detecting edge.

According  to Case 0,  we will consider hereafter that $x$ and $ab$ do not lie  simultaneously in
$\partial G$.

\noindent{\bf Case 1: \textit{m}=3.} Since three or more vertices of degree 3 lead to the
triangulated disk, $G$ is isomorphic to the wheel graph of  $4$ radii, i.e. a flag, and we reach
statement 2 in the Lemma.

Next we deal with the case $\sharp Min = 2.$ If $Min$ consists of two  adjacent vertices then one
readily  finds a flag centered at $x$. Otherwise if two  vertices in $Min$ are not adjacent, the
only possibility is that $x\in\partial G$. Let $v\in Min$  and let us consider $link(v)=\{x, t_1,
t_2\}$ with $vt_2\in \partial G$.  If $t_1\in G-\partial G$ and $deg(t_1)\geq 4$, then $vt_1$ is a
triode detecting edge. If $t_1\in \partial G$, then, by hypothesis, $deg(t_1)\geq 4$ and $xv$ is
contractible and so  a $4c$-edge if $deg(t_1)\geq 5$ or a triode detecting edge if $deg(t_1)=4$.

Finally, assume   $Min =\{v\}$ reduces to  a single vertex and   $V(link(v))=\{x, t_1, t_2\}$. We
consider the two following cases.

\noindent{\bf Case 1.1: $x$ is independent of degree 4.}

Since  there are three vertices $u\in V(link(x))$ verifying  $deg(u)\geq 5$,    Lemma
\ref{previous} shows that $xv$ is a $4c$-edge  unless $x$ and $ab$ lie in the boundary of $G$
simultaneously and $x_1x_2$ is an inner edge of $G$ (otherwise $ab$ is a triode detecting edge, as
pointed out in Case 0).

\noindent{\bf Case 1.2: $x$ is adjacent to some vertex of degree 4.}
\begin{itemize}

\item[1.2.1] If $x\in G-\partial G$, then $vt_i\in \partial G$ for $i=1,2$ and $xv$ is a contractible edge.
Moreover, $xv$ is a $4c$-edge whenever $deg(t_i)\geq 5$ and it is a
 triode detecting otherwise.

\item[1.2.2] If $x\in \partial G$ there are two  possibilities:

a) $v\in \{a, b\}$. Let us suppose $v=a$ (analogously for $v=b$). In this case  $vt_2=at_2\in
\partial G$ and $b\in G-\partial G$. Now, $ab$ is triode detecting edge if $deg(t_2)\geq
4$ and $xa$ is a $4c$-edge if $deg(t_2)=3$ (since $t_2\neq x_1$ and $deg(t_2)=3$ implies
$deg(b)\geq 5$).

b) $v\in \{x_1, x_2\}$.  If $v=x_1$ (analogously for $v=x_2$) then $vt_2=x_1t_2\in
\partial G$. In this case, $ax_1$ is triode detecting edge if $deg(t_2)\geq 4$ and $xx_1$ is a
$4c$-edge if $deg(t_2)=3$ (since $t_2\neq x_2$ and $deg(t_2)=3$ implies $deg(a)\geq 5$).

\end{itemize}

\noindent \textbf{Case 2: m $\geq$ 4.}  If $m\geq 5$, Lemma \ref{previous}.(4)  and the assumption
after Case 0 yield that a $4c$-edge incident in $x$ must appear. The same situation occurs when
$m=4$ and  $ \sharp Min = 1$.

Thus Case 2 reduces to  $m= 4$ and $ \sharp Min \geq 2$. Let  $u,v \in Min$ two  distinct vertices.
We will study the following possibilities according to the positions of the vertex $x$ and the edge
$ab$ with respect to  $\partial G$:
\begin{itemize}

\item [2.1] \textbf{$x\in \partial G$ and $a$ and $b$ are inner vertices }(and hence, $ab$ is an inner edge).
Observe that  $x_i\in \partial G$ for $i=1, 2$ since $deg(x)=4$. Let $aby$ be the other face
sharing $ab$ with $abx$.

If  $y$ is a boundary vertex or else is an inner vertex of degree at least 5,  then $ab$ is a
triode detecting edge, by definition. Let us study the case when $y$ is an inner vertex of degree
4.
\begin{itemize}
 \item [2.1.1]  Suppose that $uv$ is an edge. Then $xuv$ is a 3-cycle and, moreover, by Lemma \ref{previous}
$uv \ne x_2a$, $x_1b$ for the vertices $u, v \in Min$ chosen above.
 If $x_1x_2$ is an edge,
 an octahedron centered at $xuv$  is found. If $x_1x_2$ is not an edge the possibilities of the edge $uv$ are:

1) $uv=ab$ then an octahedron centered at $yab$ appears.

2) $uv=x_1a$ (analogously $uv=x_2b$). We can assume $b\notin Min$ ($a\notin Min$, respectively)
since, otherwise, we are in, previous subcase 1). Then the edge $xx_1$ ($xx_2$, respectively) is
4-contractible.

\item [2.1.2] Suppose that $uv$ is not an edge  (that is $uxv$ is an arc). Notice that if $\sharp Min \geq 3$ then
  at least two vertices in $Min$ form a face with $x$, and we are in case 2.1.1.
  Thus we can assume $Min=\{u,v\}$ and Lemma \ref{previous} (2) yields that either $xb$ or $xa$ is
  a $4c$-edge.

\end{itemize}

\item[2.2] \textbf{$x,$ $a$ and $b$ are inner vertices.}
\begin{itemize}
 \item [2.2.1]
 Suppose that $uv$ is an edge  (or, equivalently, $xuv$ is a face). If, in addition, $u$ and $v$ are  inner vertices
 (in particular, if $uv=ab$) then an octahedron centered at $xuv$ is found. Thus we can assume that $\{a,b\}\nsubseteq Min$ and
 $\{u, v\}\cap \partial G\neq \emptyset$.

 1) If $uv=x_1a$, as $b\notin Min$ and $x_1\in \partial G$, we easily check that  $xa$ is a triode detecting edge.
 Similarly, if $uv=x_2b$ ($x_2 \in \partial G$),
 $xb$ turns to be a triode detecting edge.

 2) If $uv=x_1x_2$, we can assume $Min=\{x_1, x_2\}$ (otherwise we are in one of the previous
 situations) and then $xx_1$ or $xx_2$ are triode detecting edges. Recall $\{x_1, x_2\}\cap\partial
 G\neq \emptyset$.

\item [2.2.2] Suppose that $uv$ is not an edge. The same arguments as in 2.1.2 reduces this case to  2.2.1 if $\sharp Min\geq 3$ or,
otherwise, $Min=\{u,v\}$ and Lemma \ref{previous} (2) yields that either $xb$ or $xa$ is
  a $4c$-edge.

\end{itemize}

\item[2.3] \textbf{$x$ is an inner vertex  and $ab$ inner edge at distance 0 from $\partial G$} (that is, precisely  $a$  or $b$ (but not both) lies
in $\partial G$). Let us suppose $a\in \partial G$ and $b$ an inner vertex (thus $xb$ is an inner
edge). From Lemma \ref{previous} (2), $xb$ is a contractible  edge. Moreover, if
 $\{a, x_2\}\cap Min = \emptyset$, then $xb$ is a $4c$-edge. Otherwise,  $\{a, x_2\}\cap Min \neq \emptyset$ and this
  case reduces to previous cases. Indeed, if $a \in Min$, we are in
   case 2.1. with $a$ playing the role of $x$. Similarly, if
 $x_2\in Min$, $x_2$ can play  the role of $x$ in case 2.1 and 2.2
 when  $x_2\in \partial G$ and $x_2\notin \partial G$, respectively.

\item[2.4] \textbf{$x\in \partial G$ and $ab$ is an  inner edge at distance 0 from $\partial G$} (that is  $\partial G\cap \{a,b\}$ reduces to a vertex).
 Assume  $a  \in \partial G$ (the case $b  \in \partial G$ is analogous) and let
 $aby$ be the other face containing $ab$. If $ab$ is a triode detecting, we are done; otherwise
(from definition of triode detecting edge) $y$ must be an inner vertex with $deg(y)=4$. This case
was studied just in the case 2.3 by interchanging $x$ and $y$.

\item[2.5] \textbf{ $x$ is an  inner vertex and $ab\in \partial G$}.   From Lemma \ref{previous}(2) $xa$ and
$xb$ are contractible edges. Since $m=4$, we are in case 2.3 when $x_i\in Min-\partial G$,   for
some $i=1,2$,  with  $xa$ or $xb$ playing the role of $ab$, for $i=1,2$, respectively. Otherwise,
we are in case 2.4. and $xa$ (analogously $xb$) playing the role of $ab$ and the role of $x$ is
played by $u$ where $u\in Min -\{a\}$.
\end{itemize}

\begin{flushright}$\Box$
\end{flushright}

The proof of  Lemma \ref{Config_6}  only  deals with the occurrences of triode detecting edges in
the proof of  Lemma \ref{Config_triode_flag2} as explained below. In Table \ref{tablalema} we
summarize the conditions under which a triode detecting edges are located in  case 2 of  the proof
of Lemma \ref{Config_triode_flag2}.

\begin{table}[h]
\begin{center}
\includegraphics[width=1\textwidth]{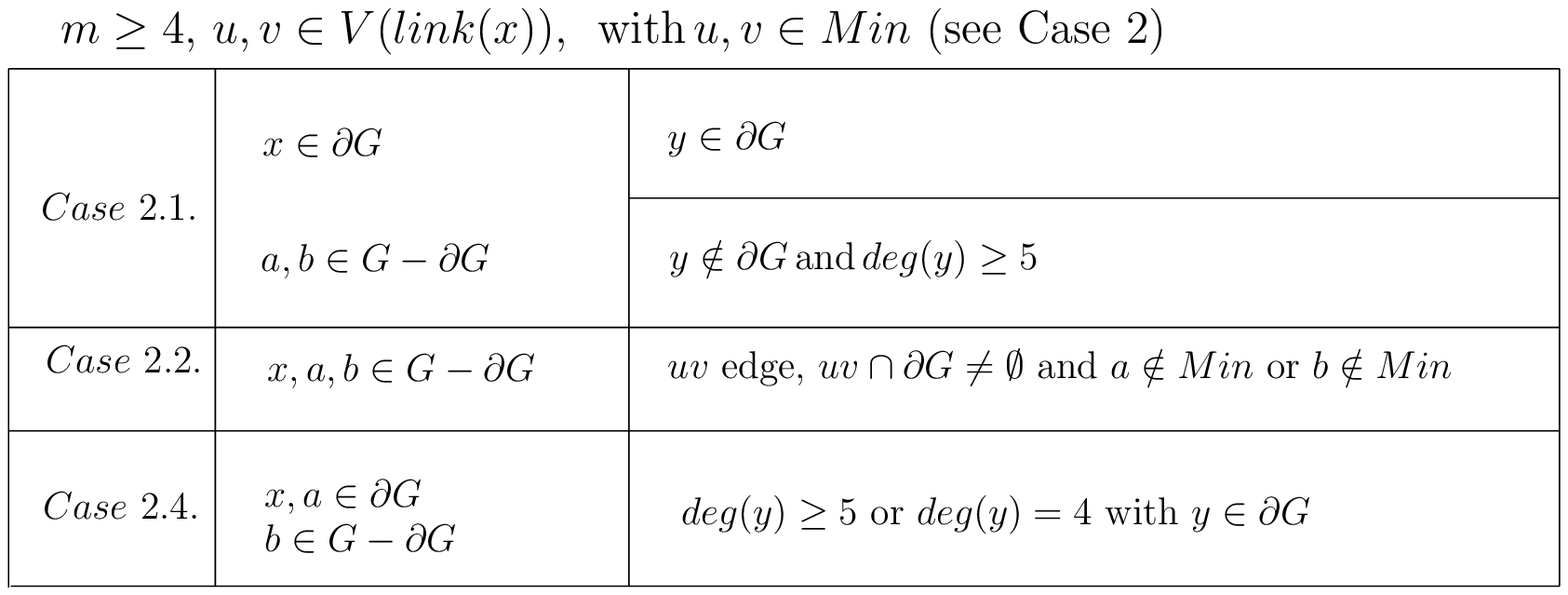}
\end{center}
\caption{Occurrences of triode detecting edges in the proof of Case 2 of Lemma
\ref{Config_triode_flag2}.}\label{tablalema}
\end{table}

\par
\noindent{\bf Proof of Lemma \ref{Config_6}}
\par

Let us start by fixing a triangulation $G\in \mathcal{F}^2(4)$ of the surface $F^2$. As
$\mathcal{F}^2(4)\subseteq \mathcal{F}_0^2(4)$, we can follow the pattern  of the  proof of Lemma
\ref{Config_triode_flag2} above. Since $G$ does not contain flags, only  the cases in the proof of
Lemma \ref{Config_triode_flag2}  when a triode detecting edge appears require a deeper analysis
(otherwise the same arguments as in the proof of Lemma \ref{Config_triode_flag2} work).  Recall
that triode detecting edges appear only in Case 0 and Case 2 of that proof. Occurrences  in Case 2
are described in Table \ref{tablalema}.

{\bf Case 0} $x \in \partial G$ and $ab \subset \partial G$

{\bf Case 0.1 m=4 and  the edge $x_1x_2$ does  not exist}. If $a\notin Min$, then the edge $xx_1$
turns to  be a $4c$-edge. Otherwise, there is  an $N$-component with parallel edges $xx_1, \, ab.$

{\bf Case 0.2 m=4 and  the edge $x_1x_2$  exists}  (and it is necessarily an inner edge). Observe
that $Min=\{a, b, x_1, x_2\}$ yields that $G$ is necessarily the irreducible triangulation $M_1$
($K_6$ minus a vertex) of the M\"{o}bius strip given in \cite{CLQV2014}. This  contradicts that
 $ab$ is contractible in $G$. Similarly, the contractibility of $ab$ implies that  the edges $x_2a$
 and $x_1b$ do not exist in $G$ (Lemma \ref{previous}). Hence an $M$-component centered at
 $abx$ is found and Remark \ref{lem:M} assures that $\{x_1, x_2\}\cap Min=\emptyset$

{\bf Case 0.3 m=5 and  the edge $x_1x_2$ does not  exist.} Then $xx_i$ is a $4c$-edge for $i=1,2$.

{\bf Case 0.4 m=5 and  the edge $x_1x_2$  exists.} In this case,  an $M$-configuration centered at
$abx$ is found and  the proof  for the Case 0 is finished.

{\bf Case 2 m=4 } See Table \ref{tablalema}; recall that $Min$ contains at least two vertices $u, v
\in V(link(x))$.

{\bf Case 2.1} $x\in \partial G, \, a, b\in G-\partial G$. This implies $x_1, x_2 \in \partial G$.
Let $aby$ be the other face sharing $ab$ with $abx$.

{\bf Case 2.1.a)} $y\in \partial G$ and then $ab$ is a triode detecting edge (see Definition
\ref{def:triode}).

\begin{enumerate}
\item Assume $deg(y)=4$. If $a\in Min$, then $yx_1\in \partial G$ and  $deg(x_1)=3,$  which contradicts
the hypothesis. (Analogously, $b\in Min$ leads to  $deg(x_2)=3$). Hence $\{u, v\}=\{x_1, x_2\}$.
Notice that $x_1x_2$ is not a boundary edge since otherwise $\partial G =x_1xx_2$, which leads to a
contradiction with $y\in \partial G$. Moreover, as $x_1, x_2\in Min$, if $x_1x_2$ is an inner edge
there exists the face $x_1ax_2$ which  contradicts the contractibility of $ab$. Therefore, the edge
$x_1x_2$ does not exist in $G$ and hence $xx_i$ is a $4c$-edge for $i=1,3$ by Lemma \ref{previous}.

\item Assume  $deg(y)\geq5$. It is clear that in case that the edge $x_1x_2$ exists, it cannot be a
boundary edge. If  $\{u, v\}=\{a, b\}$,  then  a quasi-octahedron centered at $abx$ and remaining
vertices $\{x_1, x_2, y\}$ is found. Suppose now $a\notin Min$. If $x_1x_2$ is not an edge, then by
Lemma \ref{previous} $xx_1$ is a $4c$-edge. If $x_1x_2$ is an edge, then there must be $x_2\notin
Min$ (otherwise $bx_1$ is an edge contradicting the hypothesis of $ab$ contractible edge) and then
by Lemma \ref{previous} $xb$ is a $4c$-edge.
\end{enumerate}

{\bf Case 2.1.b)} $y\in G-\partial G$, $deg(y)\geq 5 $  (recall that  $ab$ is a triode detecting
edge).

Firstly, we consider  $\{u, v\}=\{a, b\}$. If $x_1x_2$ is not an edge, clearly an $N$-configuration
with parallel edges  $ab,\, xx_1$ is detected. (Moreover, $xab$ is the center of a quasi-octahedron
component with remaining vertices $\{x_1, x_2, y\}$). If $x_1x_2$ is a boundary edge, then an
octahedron centered at $abx$ appears. If $x_1x_2$ is an inner edge, a quasi-octahedron component
centered at $abx$ and remaining vertices $\{x_1, x_2, y\}$ is found.

Secondly, we consider  $\{u, v\}=\{x_1, x_2\}$.  Notice that $x_1x_2$ cannot be  an inner edge
since otherwise, as $x_1\in
\partial G$ and $x_1\in Min$, there should be a boundary vertex $p$ defining a face $x_1x_2p$, and
then $x_2a$ should be an edge contradicting the fact that $ab$ is a contractible edge. Moreover, if
  $x_1x_2$ is a boundary edge, then an octahedron
centered at $xx_1x_2$ is found. It remains to  consider that $x_1x_2$ is not an edge. Then there is
an $N$-component with parallel edges  $ab,\, xx_1$.

Next, we consider  $\{u, v\}=\{a, x_2\}$. Here  $b\notin Min$ and then $xa$ is a $4c$-edge by Lemma
\ref{previous}. (Observe that  $b\in Min$ corresponds to  the first subcase studied above). An
analogous argument works for $\{u, v\}=\{b, x_1\}$.

Finally, we consider  $\{u, v\}=\{a, x_1\}$. Observe that in case the edge $x_1x_2$ exists, it can
not be an inner edge since $x_1\in Min.$ If  $x_1x_2$ is a boundary edge, then an octahedron
centered at $xx_1a$ is found.  In case that  $x_1x_2$ is not an edge, we get that $xx_2$ is a
4-contractible edge by Lemma \ref{previous}.
 An analogous argument
works for $\{u, v\}=\{b, x_2\}$.

 {\bf Case 2.2} $x, \, a, b\in G-\partial G$.

 According to  Table \ref{tablalema} we have to study only  the case when $uv$ is an edge  (hence $xuv$ is a
triangle) and $u\in \partial G$ or $v\in \partial G$, whence $xv$ or $xu$ is a triode detecting
edge.

 Moreover, if $\{u, v\}\cap\{x_1, x_2\}\neq\emptyset$ we distinguish two  possibilities.

 a)  If $\{u, v\}\cap\{x_1, x_2\}=\{x_1\}$ (or $x_2$),
then $v\in \{a, b\}$.

If $uv=x_1a$ (similarly $uv=x_2b$), then  Lemma \ref{previous} yields that $xa$ is  contractible.
Now it is not difficult to see that this case reduces to Case 2.1 above: $xa$ plays the same role
as $ab$, $x_1\in
\partial G$ plays the role of  $x$ and $b$ plays the role of $y$.

b) If $uv=x_1x_2$, then $x_1x_2$ is necessarily  a boundary  edge and $xu$ (and $xv$) is a
$cn4c$-edge. This case reduces to  Case 2.4 in Lemma \ref{Config_triode_flag2}: $xu$ plays the same
role as $ab$ ($v$ plays the role of $x$).

{\bf Case 2.4} $x, a\in \partial G, \, b\in G-\partial G$. It readily  follows that  in this case
$x_1x$ is a boundary edge (here we use $deg(x)=4$) and $x_1x_2$ is an inner edge (otherwise
$deg(x_1)=2$).
\begin{enumerate}
\item First we analyze  the case  $uv=ab$. If  $ax_1$ is an edge, then an octahedron
 centered at $abx$ is located. If $ax_1$ is not an edge and $deg(x_2)\geq 5$, then an
 $N$-component with parallel edges $ax$ and $bx_2$ is found. If $ax_1$ is not an edge and
 $deg(x_2)=4$, then a quasi-octahedron component centered at $abx_2$ and remaining vertices $x_1, a, y$ is
detected (observe that $y$ must be an inner vertex, otherwise a boundary vertex of degree 2
appears).

\item Notice that when $uv=x_1x_2$ we have $x_i\in \partial G$ and $x_2\notin \partial G$ and
$x_1x_2$ is a $cn4c$-edge. Then a similar argument as above  works with $x_1x_2$ playing the role
of  $ab$ to obtain either an $N$-component or a quasi-octahedron component.

It remains to  deal with the subcase when  the sets $\{u,v\}\neq \{a,b\}$ share exactly one
element.

\item If $u=a$  and $v=x_1$ and $ax_1$ is an edge, then an octahedron
 centered at $axx_1$ is found and $\partial G=axx_1$. If $ax_1$ is not an edge, as $deg(b)\geq 5$,  then  $ax$
 is a $4c$-edge.

\item  If $u=a$, $v=x_2$ then $deg(x_1)\geq 5$ and $deg(b)\geq 5$ and we
 conclude that by  Lemma
\ref{previous} $xx_2$ is a $4c$-edge. Observe that $x_2$ is an inner vertex since $xx_1 \subset
\partial G$,  $deg(x_2)= 4$ and $b\notin \partial G$.

\item If $u=b$, $v=x_1$, necessarily $deg(x_2),\,  deg(a)\geq 5$ and so $xb$ is a $4c$-edge by  Lemma
\ref{previous}.

\item If $u=b$, $v=x_2$  then $x_1x_2y$ is necessarily a  face. Moreover, if $ax_1$ is a boundary
edge, then an octahedron  component centered at $bxx_2$ is found and $\partial G=axx_1$.

 On the other hand, if $ax_1$ is an inner edge then  an octahedron component centered at
$byx_2$ appears  whenever  $ax_1y$ is a face (so  that  $deg(y)=4$), if it is not a face then a
quasi-octahedron component centered at $bxx_2$ and remaining vertices $x_1, a, y$ is
 found.

 Finally, the same quasi-octahedron component is located in $G$ when $ax_1$ is not an edge.
This finishes the proof of Lemma \ref{Config_6}.
\end{enumerate}
\begin{flushright}$\Box$
\end{flushright}

\bibliographystyle{spmpsci}

\end{document}